\documentclass{amsart}
\usepackage{amssymb,latexsym}
\usepackage{amsrefs}
\usepackage[usesnames,svgnames]{xcolor}
\usepackage{graphicx}
\usepackage{tikz,pgf}
\usepackage[colorlinks=true,
	    urlcolor=MidnightBlue,
	    citecolor=DarkGreen]{hyperref}
\pgfdeclarelayer{background}
\pgfsetlayers{background,main}
\newtheorem{theorem}{Theorem}[section]

\newtheorem{proposition}[theorem]{Proposition}

\newtheorem{lemma}[theorem]{Lemma} 

\theoremstyle{remark}
\newtheorem{remark}[theorem]{Remark}
\newtheorem{example}[theorem]{Example}
\author{Myrto Kallipoliti}
\author{Henri M\"uhle}
\address{Fak. f\"ur Mathematik, Universit\"at Wien, Garnisongasse 3, 1090 Wien, Austria}
\email{myrto.kallipoliti@univie.ac.at}
\email{henri.muehle@univie.ac.at}
\thanks{This work was funded by the FWF research grant no. Z130-N13.}
\title{On the Topology of the Cambrian Semilattices}
\newcommand{\cambCongA}[2]{
	\begin{tikzpicture}\scriptsize
		\def\x{#1};
		\def\y{#2};
		\draw(3.5*\x,1.25*\y) node(n0){$\varepsilon$};
		\draw(2.5*\x,2*\y) node(n1){$s_{1}$};
		\draw(3.5*\x,2*\y) node(n2){$s_{2}$};
		\draw(4.5*\x,2*\y) node(n3){$s_{3}$};
		\draw(1.5*\x,3*\y) node(n4){$s_{1}s_{2}$};
		\draw(2.5*\x,3*\y) node(n5){$s_{2}\vert s_{1}$};
		\draw(3.5*\x,3*\y) node(n6){$s_{1}s_{3}$};
		\draw(4.5*\x,3*\y) node(n7){$s_{2}s_{3}$};
		\draw(5.5*\x,3*\y) node(n8){$s_{3}\vert s_{2}$};
		\draw(1*\x,4*\y) node(n9){$s_{1}s_{2}\vert s_{1}$};
		\draw(2*\x,4*\y) node(n10){$s_{1}s_{2}s_{3}$};
		\draw(3*\x,4*\y) node(n11){$s_{1}s_{3}\vert s_{2}$};
		\draw(4*\x,4*\y) node(n12){$s_{2}s_{3}\vert s_{1}$};
		\draw(5*\x,4*\y) node(n13){$s_{3}\vert s_{2}\vert s_{1}$};
		\draw(6*\x,4*\y) node(n14){$s_{2}s_{3}\vert s_{2}$};
		\draw(1.5*\x,5*\y) node(n15){$s_{1}s_{2}s_{3}\vert s_{1}$};
		\draw(2.5*\x,5*\y) node(n16){$s_{1}s_{2}s_{3}\vert s_{2}$};
		\draw(3.5*\x,5*\y) node(n17){$s_{1}s_{3}\vert s_{2}\vert s_{1}$};
		\draw(4.5*\x,5*\y) node(n18){$s_{2}s_{3}\vert s_{1}s_{2}$};
		\draw(5.5*\x,5*\y) node(n19){$s_{2}s_{3}\vert s_{2}\vert s_{1}$};
		\draw(2.5*\x,6*\y) node(n20){$s_{1}s_{2}s_{3}\vert s_{1}s_{2}$};
		\draw(3.5*\x,6*\y) node(n21){$s_{1}s_{2}s_{3}\vert s_{2}\vert s_{1}$};
		\draw(4.5*\x,6*\y) node(n22){$s_{2}s_{3}\vert s_{1}s_{2}\vert s_{1}$};
		\draw(3.5*\x,6.75*\y) node(n23){$s_{1}s_{2}s_{3}\vert s_{1}s_{2}\vert s_{1}$};
		\draw(n0) -- (n1);
		\draw(n0) -- (n2);
		\draw(n0) -- (n3);
		\draw(n1) -- (n4);
		\draw(n1) -- (n6);
		\draw(n2) -- (n5);
		\draw(n2) -- (n7);
		\draw(n3) -- (n6);
		\draw(n3) -- (n8);
		\draw(n4) -- (n9);
		\draw(n4) -- (n10);
		\draw(n5) -- (n9);
		\draw(n5) -- (n12);
		\draw(n6) -- (n11);
		\draw(n7) -- (n12);
		\draw(n7) -- (n14);
		\draw(n8) -- (n13);
		\draw(n8) -- (n14);
		\draw(n9) -- (n15);
		\draw(n10) -- (n15);
		\draw(n10) -- (n16);
		\draw(n11) -- (n16);
		\draw(n11) -- (n17);
		\draw(n12) -- (n18);
		\draw(n13) -- (n17);
		\draw(n13) -- (n19);
		\draw(n14) -- (n19);
		\draw(n15) -- (n20);
		\draw(n16) -- (n21);
		\draw(n17) -- (n21);
		\draw(n18) -- (n20);
		\draw(n18) -- (n22);
		\draw(n19) -- (n22);
		\draw(n20) -- (n23);
		\draw(n21) -- (n23);
		\draw(n22) -- (n23);
		\begin{pgfonlayer}{background}
		  \draw[line width=35pt,line cap=round,rounded corners,blue!40!white](3.5*\x,2*\y) -- 
		    (2.5*\x,3*\y);
		  \draw[line width=35pt,line cap=round,rounded corners,blue!40!white](4.5*\x,3*\y) -- 
		    (4*\x,4*\y) -- (4.5*\x,5*\y);
		  \draw[line width=35pt,line cap=round,rounded corners,blue!40!white](6*\x,4*\y) -- 
		    (5.5*\x,5*\y) -- (4.5*\x,6*\y);
		  \draw[line width=35pt,line cap=round,rounded corners,blue!40!white](4.5*\x,2*\y) -- 
		    (5.5*\x,3*\y) -- (5*\x,4*\y);
		  \draw[line width=35pt,line cap=round,rounded corners,blue!40!white](3.5*\x,3*\y) -- 
		    (3*\x,4*\y) -- (3.5*\x,5*\y);
		  \draw[line width=35pt,line cap=round,rounded corners,blue!40!white](2.5*\x,5*\y) -- 
		    (3.5*\x,6*\y);
		\end{pgfonlayer}
\end{tikzpicture}}

\newcommand{\cambAThreeSkeleton}[2]{
	\def\x{#1};
	\def\y{#2};
	\draw(3*\x,1*\y) node(v0){$\varepsilon$};	
	\draw(2*\x,2*\y) node(v1){$s_{1}$};
	\draw(3*\x,2*\y) node(v2){$s_{2}$};
	\draw(4*\x,2*\y) node(v3){$s_{3}$};
	\draw(1*\x,3*\y) node(v4){$s_{1}s_{2}$};
	\draw(3*\x,3*\y) node(v5){$s_{1}s_{3}$};
	\draw(4*\x,3*\y) node(v6){$s_{2}s_{3}$};
	\draw(.5*\x,4*\y) node(v7){$s_{1}s_{2}\vert s_{1}$};
	\draw(1.5*\x,4*\y) node(v8){$s_{1}s_{2}s_{3}$};
	\draw(5*\x,4*\y) node(v9){$s_{2}s_{3}\vert s_{2}$};
	\draw(1*\x,5*\y) node(v10){$s_{1}s_{2}s_{3}\vert s_{1}$};
	\draw(2*\x,5*\y) node(v11){$s_{1}s_{2}s_{3}\vert s_{2}$};
	\draw(2*\x,6*\y) node(v12){$s_{1}s_{2}s_{3}\vert s_{1}s_{2}$};
	\draw(3*\x,7*\y) node(v13){$s_{1}s_{2}s_{3}\vert s_{1}s_{2}\vert s_{1}$};
}

\newcommand{\cambAThree}[2]{
	\begin{tikzpicture}\scriptsize
		\cambAThreeSkeleton{#1}{#2}
		\draw[very thick](v0) -- (v1) node[fill=white] at(2.5*\x,1.5*\y){$1$};
		\draw(v0) -- (v2) node[fill=white] at(3*\x,1.5*\y){$2$};
		\draw(v0) -- (v3) node[fill=white] at(3.5*\x,1.5*\y){$3$};
		\draw[very thick](v1) -- (v4) node[fill=white] at(1.5*\x,2.5*\y){$2$};
		\draw(v1) -- (v5) node[fill=white] at(2.67*\x,2.67*\y){$3$};
		\draw(v2) -- (v6) node[fill=white] at(3.67*\x,2.67*\y){$3$};
		\draw(v2) -- (v7) node[fill=white] at(1.75*\x,3*\y){$1$};
		\draw(v3) -- (v5) node[fill=white] at(3.33*\x,2.67*\y){$1$};
		\draw(v3) -- (v9) node[fill=white] at(4.5*\x,3*\y){$2$};
		\draw(v4) -- (v7) node[fill=white] at(.75*\x,3.5*\y){$4$};
		\draw[very thick](v4) -- (v8) node[fill=white] at(1.33*\x,3.67*\y){$3$};
		\draw(v5) -- (v11) node[fill=white] at(2.5*\x,4*\y){$2$};
		\draw(v6) -- (v9) node[fill=white] at(4.5*\x,3.5*\y){$5$};
		\draw(v6) -- (v12) node[fill=white] at(3*\x,4.5*\y){$1$};
		\draw(v7) -- (v10) node[fill=white] at(.75*\x,4.5*\y){$3$};
		\draw[very thick](v8) -- (v10) node[fill=white] at(1.25*\x,4.5*\y){$4$};
		\draw(v8) -- (v11) node[fill=white] at(1.75*\x,4.5*\y){$5$};
		\draw(v9) -- (v13) node[fill=white] at(4*\x,5.5*\y){$1$};
		\draw[very thick](v10) -- (v12) node[fill=white] at(1.5*\x,5.5*\y){$5$};
		\draw(v11) -- (v13) node[fill=white] at(2.62*\x,6.25*\y){$4$};
		\draw[very thick](v12) -- (v13) node[fill=white] at(2.5*\x,6.5*\y){$7$};
	\end{tikzpicture}}

\newcommand{\cambBThree}[2]{
	\begin{tikzpicture}\scriptsize
		\def\x{#1};
		\def\y{#2};
		\draw(4*\x,1*\y) node(v0){$\varepsilon$};
		\draw(2*\x,2*\y) node(v1){$s_{3}$};
		\draw(4*\x,2*\y) node(v2){$s_{2}$};
		\draw(6*\x,2*\y) node(v3){$s_{1}$};
		\draw(3*\x,3*\y) node(v4){$s_{2}s_{3}$};
		\draw(4*\x,3*\y) node(v5){$s_{1}s_{3}$};
		\draw(7*\x,3*\y) node(v6){$s_{1}s_{2}$};
		\draw(2*\x,4*\y) node(v7){$s_{2}s_{3}\vert s_{2}$};
		\draw(5.67*\x,3.67*\y) node(v8){$s_{1}s_{2}\vert s_{1}$};
		\draw(8*\x,4*\y) node(v9){$s_{1}s_{2}s_{3}$};
		\draw(1*\x,5*\y) node(v10){$s_{2}s_{3}\vert s_{2}s_{3}$};
		\draw(6.67*\x,4.67*\y) node(v11){$s_{1}s_{2}s_{3}\vert s_{1}$};
		\draw(8.67*\x,4.67*\y) node(v12){$s_{1}s_{2}s_{3}\vert s_{2}$};
		\draw(6.7*\x,5.75*\y) node(v13){$s_{1}s_{2}s_{3}\vert s_{2}s_{3}$};
		\draw(8*\x,6*\y) node(v14){$s_{1}s_{2}s_{3}\vert s_{1}s_{2}$};
		\draw(7*\x,7*\y) node(v15){$s_{1}s_{2}s_{3}\vert s_{1}s_{2}s_{3}$};
		\draw(9*\x,7*\y) node(v16){$s_{1}s_{2}s_{3}\vert s_{1}s_{2}\vert s_{1}$};
		\draw(8*\x,8*\y) node(v17){$s_{1}s_{2}s_{3}\vert s_{1}s_{2}s_{3}\vert s_{1}$};
		\draw(7*\x,9*\y) node(v18){$s_{1}s_{2}s_{3}\vert s_{1}s_{2}s_{3}\vert s_{1}s_{2}$};
		\draw(4*\x,10*\y) node(v19){$s_{1}s_{2}s_{3}\vert s_{1}s_{2}s_{3}\vert s_{1}s_{2}s_{3}$};
		\draw(v0) -- (v1) node[fill=white] at (3*\x,1.5*\y){$3$};
		\draw(v0) -- (v2) node[fill=white] at (4*\x,1.5*\y){$2$};
		\draw[very thick](v0) -- (v3) node[fill=white] at (5*\x,1.5*\y){$1$};
		\draw(v1) -- (v5) node[fill=white] at (3*\x,2.5*\y){$1$};
		\draw(v1) -- (v10) node[fill=white] at (1.5*\x,3.5*\y){$2$};
		\draw(v2) -- (v4) node[fill=white] at (3.5*\x,2.5*\y){$3$};
		\draw(v2) -- (v8) node[fill=white] at (5.17*\x,3.17*\y){$1$};
		\draw(v3) -- (v5) node[fill=white] at (5.33*\x,2.33*\y){$3$};
		\draw[very thick](v3) -- (v6) node[fill=white] at (6.5*\x,2.5*\y){$2$};
		\draw(v4) -- (v7) node[fill=white] at (2.5*\x,3.5*\y){$5$};
		\draw(v4) -- (v15) node[fill=white] at (5*\x,5*\y){$1$};
		\draw(v5) -- (v13) node[fill=white] at (5.5*\x,4.5*\y){$2$};
		\draw(v6) -- (v8) node[fill=white] at (6.33*\x,3.33*\y){$4$};
		\draw[very thick](v6) -- (v9) node[fill=white] at (7.5*\x,3.5*\y){$3$};
		\draw(v7) -- (v10) node[fill=white] at (1.5*\x,4.5*\y){$6$};
		\draw(v7) -- (v18) node[fill=white] at (4.5*\x,6.5*\y){$1$};
		\draw(v8) -- (v11) node[fill=white] at (6.17*\x,4.17*\y){$3$};
		\draw[very thick](v9) -- (v11) node[fill=white] at (7.33*\x,4.33*\y){$4$};
		\draw(v9) -- (v12) node[fill=white] at (8.33*\x,4.33*\y){$5$};
		\draw(v10) -- (v19) node[fill=white] at (2.5*\x,7.5*\y){$1$};
		\draw[very thick](v11) -- (v14) node[fill=white] at (7.17*\x,5.17*\y){$5$};
		\draw(v12) -- (v13) node[fill=white] at (8*\x,5.05*\y){$6$};
		\draw(v12) -- (v16) node[fill=white] at (8.8*\x,5.75*\y){$4$};
		\draw(v13) -- (v19) node[fill=white] at (5.25*\x,8*\y){$4$};
		\draw[very thick](v14) -- (v15) node[fill=white] at (7.5*\x,6.5*\y){$6$};
		\draw(v14) -- (v16) node[fill=white] at (8.5*\x,6.5*\y){$7$};
		\draw[very thick](v15) -- (v17) node[fill=white] at (7.5*\x,7.5*\y){$7$};
		\draw(v16) -- (v17) node[fill=white] at (8.5*\x,7.5*\y){$6$};
		\draw[very thick](v17) -- (v18) node[fill=white] at (7.5*\x,8.5*\y){$8$};
		\draw[very thick](v18) -- (v19) node[fill=white] at (5.5*\x,9.5*\y){$9$};
	\end{tikzpicture}}

\newcommand{\cambAffineATwo}[2]{
	\begin{tikzpicture}\scriptsize
		\def\x{#1};
		\def\y{#2};
		\draw(4*\x,1*\y) node(v1){$\varepsilon$};
		\draw(1*\x,2*\y) node(v2){$s_{3}$};
		\draw(5*\x,2*\y) node(v3){$s_{1}$};
		\draw(8*\x,2*\y) node(v4){$s_{2}$};
		\draw(2*\x,3*\y) node(v5){$s_{1}s_{3}$};
		\draw(6.25*\x,3*\y) node(v6){$s_{2}s_{3}$};
		\draw(7.25*\x,3*\y) node(v7){$s_{1}s_{2}$};
		\draw(.5*\x,4*\y) node(v8){$s_{1}s_{3}\vert s_{1}$};
		\draw(4*\x,4*\y) node(v9){$s_{2}s_{3}\vert s_{2}$};
		\draw(6*\x,4*\y) node(v10){$s_{1}s_{2}s_{3}$};
		\draw(9.5*\x,4*\y) node(v11){$s_{1}s_{2}\vert s_{1}$};
		\draw(2.5*\x,5*\y) node(v12){$s_{1}s_{2}s_{3}\vert s_{2}$};
		\draw(7.5*\x,5*\y) node(v13){$s_{1}s_{2}s_{3}\vert s_{1}$};
		\draw(6*\x,6*\y) node(v14){$s_{1}s_{2}s_{3}\vert s_{1}s_{2}$};
		\draw(9.5*\x,6*\y) node(v15){$s_{1}s_{2}s_{3}\vert s_{1}s_{3}$};
		\draw(4*\x,7*\y) node(v16){$s_{1}s_{2}s_{3}\vert s_{1}s_{2}\vert s_{1}$};
		\draw(7*\x,7*\y) node(v17){$s_{1}s_{2}s_{3}\vert s_{1}s_{2}s_{3}$};
		\draw(5*\x,8*\y) node(v18){$s_{1}s_{2}s_{3}\vert s_{1}s_{2}s_{3}\vert s_{1}$};
		\draw(8*\x,8*\y) node(v19){$s_{1}s_{2}s_{3}\vert s_{1}s_{2}s_{3}\vert s_{2}$};
		\draw(v1) -- (v2) node[fill=white] at (2.5*\x,1.5*\y){$3$};
		\draw(v1) -- (v3) node[fill=white] at (4.5*\x,1.5*\y){$1$};
		\draw(v1) -- (v4) node[fill=white] at (6*\x,1.5*\y){$2$};
		\draw(v2) -- (v8) node[fill=white] at (.75*\x,3*\y){$1$};
		\draw(v2) -- (v9) node[fill=white] at (2.9*\x,3.25*\y){$2$};
		\draw(v3) -- (v5) node[fill=white] at (3.5*\x,2.5*\y){$3$};
		\draw(v3) -- (v7) node[fill=white] at (6.125*\x,2.5*\y){$2$};
		\draw(v4) -- (v6) node[fill=white] at (7.125*\x,2.5*\y){$3$};
		\draw(v4) -- (v11) node[fill=white] at (8.75*\x,3*\y){$1$};
		\draw(v5) -- (v8) node[fill=white] at (1.25*\x,3.5*\y){$4$};
		\draw(v5) -- (v12) node[fill=white] at (2.25*\x,4*\y){$2$};
		\draw(v6) -- (v9) node[fill=white] at (5.125*\x,3.5*\y){$5$};
		\draw(v7) -- (v10) node[fill=white] at (6.625*\x,3.5*\y){$3$};
		\draw(v7) -- (v11) node[fill=white] at (8.375*\x,3.5*\y){$4$};
		\draw(v10) -- (v12) node[fill=white] at (4.25*\x,4.5*\y){$5$};
		\draw(v10) -- (v13) node[fill=white] at (6.75*\x,4.5*\y){$4$};
		\draw(v11) -- (v15) node[fill=white] at (9.5*\x,5*\y){$3$};
		\draw(v12) -- (v16) node[fill=white] at (3.25*\x,6*\y){$4$};
		\draw(v13) -- (v14) node[fill=white] at (6.75*\x,5.5*\y){$5$};
		\draw(v13) -- (v15) node[fill=white] at (8.5*\x,5.5*\y){$6$};
		\draw(v14) -- (v16) node[fill=white] at (5*\x,6.5*\y){$7$};
		\draw(v14) -- (v17) node[fill=white] at (6.5*\x,6.5*\y){$6$};
		\draw(v15) -- (v19) node[fill=white] at (8.75*\x,7*\y){$5$};
		\draw(v17) -- (v18) node[fill=white] at (6*\x,7.5*\y){$7$};
		\draw(v17) -- (v19) node[fill=white] at (7.5*\x,7.5*\y){$8$};
	\end{tikzpicture}
}

\begin{document}

\begin{abstract}
	For an arbitrary Coxeter group $W$, David Speyer and Nathan Reading defined Cambrian semilattices 
	$C_{\gamma}$ as semilattice quotients of the weak order on $W$ induced by certain semilattice 
	homomorphisms. In this article, we define an edge-labeling using the realization of Cambrian 
	semilattices in terms of $\gamma$-sortable elements, and show that this is an EL-labeling for every
	closed interval of $C_{\gamma}$. In addition, we use our labeling to show that every finite open
	interval in a Cambrian semilattice is either contractible or spherical, and we characterize 
	the spherical intervals, generalizing a result by Nathan Reading.
\end{abstract}

\maketitle

\section{Introduction}
  \label{sec:introduction}
In \cite[Theorem~9.6]{bjorner97shellable} Anders Bj\"orner and Michelle Wachs showed that the Tamari 
lattice $T_{n}$, introduced in \cite{tamari62algebra}, can be regarded as the subposet of the weak-order
lattice on the symmetric group $\mathfrak{S}_{n}$, consisting of 312-avoiding permutations. More precisely,
there exists a lattice homomorphism $\sigma:\mathfrak{S}_{n}\to T_{n}$ such that $T_{n}$ is isomorphic
to the subposet of the weak-order lattice on $\mathfrak{S}_{n}$ consisting of the bottom elements in the
fibers of $\sigma$. In \cite{reading06cambrian}, the map $\sigma$ was realized as a map from
$\mathfrak{S}_{n}$ to the triangulations of an $(n+2)$-gon, where the partial order on the latter is given
by diagonal flips. It was shown that the fibers of $\sigma$ induce a congruence relation on the weak-order
lattice on $\mathfrak{S}_{n}$, and that the Tamari lattice is isomorphic to the lattice quotient induced by
this congruence. Moreover, it was observed that different embeddings of the $(n+2)$-gon in the plane yield
different lattice quotients of the weak-order lattice on $\mathfrak{S}_{n}$. The realization of
$\mathfrak{S}_{n}$ as the Coxeter group $A_{n-1}$ was then used to connect the embedding of the $(n+2)$-gon
in the plane with a Coxeter element of $A_{n-1}$. This connection eventually led to the definition of
Cambrian lattices, which can analogously be defined for an arbitrary finite Coxeter group $W$ as lattice
quotients of the weak-order lattice on $W$ with respect to certain lattice congruences induced by
orientations of the Coxeter diagram of $W$ (see \cite{reading07sortable}). 

As suggested in \cite{stasheff97from}*{Appendix~B}, and later in \cite{loday04realization}*{Theorem~1},
the Hasse diagram of the Tamari lattice corresponds to the $1$-skeleton of the classical associa\-hedron.
(Due to the connection to the symmetric group, which was elaborated in \cite{loday04realization}, the
classical associahedron is also referred to as \emph{type $A$-associahedron}.) In 
\cite{bott94on, simion03type, fomin03systems, chapoton02polytopal}, generalized associahedra 
were defined for all crystallographic Coxeter groups which generalize the type $A$-associahedron. The 
Cambrian lattices provide another viewpoint for the generalized associahedra, namely that the fan
associated to a Cambrian lattice of crystallographic type is the normal fan of the generalized
associahedron of the same type (see \cite{reading09cambrian} for the details of this construction).
Moreover, since the Cambrian lattices are defined for all finite Coxeter groups, this connection defines a
generalized associahedron for the non-crystallographic types as well (see
\cite{reading09cambrian}*{Corollary~8.1}). 

In \cite{reading11sortable}, Nathan Reading and David Speyer generalized the construction of Cambrian 
lattices to infinite Coxeter groups. Since in general, there exists no maximum element in an infinite 
Coxeter group, the weak order constitutes only a (meet)-semilattice. Using the realization of the Cambrian 
lattices in terms of Coxeter-sortable elements, which was first described in \cite{reading07sortable} and 
later extended in \cite{reading11sortable}, the analogous construction as in the finite case yields a 
quotient semilattice of the weak-order semilattice, the so-called \emph{Cambrian semilattice}.

This article is dedicated to the investigation of the topological properties of the order complex of 
the proper part of closed intervals in a Cambrian semilattice. One (order-theoretic) tool to 
investigate these properties is EL-shellability, which was introduced in \cite{bjorner80shellable}, and
further developed in \cite{bjorner83lexicographically,bjorner96shellable,bjorner97shellable}. The fact
that a poset is EL-shellable implies a number of properties of the associated order complex: this order
complex is Cohen-Macaulay, it is homotopy equivalent to a wedge of spheres and the dimensions of its 
homology groups can be computed from the labeling. The first main result of the present article is the
following.
\begin{theorem}
  \label{thm:infinite_shellability}
	Every closed interval in $C_{\gamma}$ is EL-shellable for every (possibly infinite) Coxeter group
	$W$ and every Coxeter element $\gamma\in W$.
\end{theorem}
We prove this result uniformly using the realization of $C_{\gamma}$ in terms of Coxeter-sortable elements, 
and thus our proof does not require $W$ to be finite or even crystallographic. 
For finite crystallographic Coxeter groups, Theorem~\ref{thm:infinite_shellability}
is implied by \cite[Theorem~4.17]{ingalls09noncrossing}. Colin Ingalls and Hugh Thomas considered in
\cite{ingalls09noncrossing} the category of finite dimensional representations of an orientation of the
Coxeter diagram of a finite crystallographic Coxeter group $W$, and considered the corresponding Cambrian 
lattices as a poset of torsion classes of this category. However, their approach cannot be 
applied to non-crystallographic or to infinite Coxeter groups.

Finally, using the fact that every closed interval of $C_{\gamma}$ is EL-shellable, we are able to
determine the homotopy type of the proper parts of these intervals by counting the number of falling chains
with respect to our labeling. It turns out that every open interval is either contractible or spherical,
\emph{i.e.} homotopy equivalent to a sphere. We can further characterize which intervals of $C_{\gamma}$
are contractible and which are spherical, as our second main result shows. Recall that a closed interval
$[x,y]$ in a lattice is called \emph{nuclear} if $y$ is the join of atoms of $[x,y]$.
\begin{theorem}
  \label{thm:cambrian_topology}
	Let $W$ be a (possibly infinite) Coxeter group and let $\gamma\in W$ be a Coxeter element. Every 
	finite open interval in the Cambrian semilattice $C_{\gamma}$ is either contractible or spherical. 
	Furthermore, a finite open interval $(x,y)_{\gamma}$ is spherical if and only if the corresponding 
	closed interval $[x,y]_{\gamma}$ is nuclear. 
\end{theorem}
For finite Coxeter groups, Theorem~\ref{thm:cambrian_topology} is implied by concatenating
\cite[Theorem~1.1]{reading05lattice} and \cite[Propositions~5.6~and~5.7]{reading05lattice}. Nathan
Reading's approach in the cited article was to investigate fan posets of central hyperplane arrangements.
He showed that for a finite Coxeter group $W$ the Cambrian lattices can be viewed as fan posets of a
fan induced by certain regions of the Coxeter arrangement of $W$ which are determined by orientations of
the Coxeter diagram of $W$. The tools Nathan Reading developed in \cite{reading05lattice} apply to a much
larger class of fan posets, but cannot be applied directly to infinite Coxeter groups. 

The proofs of Theorems~\ref{thm:infinite_shellability} and \ref{thm:cambrian_topology} are obtained 
completely within the framework of Coxeter-sortable elements and thus have the advantage that 
they are uniform and direct. 

\smallskip

This article is organized as follows. In Section~\ref{sec:preliminaries}, we recall the necessary
order-theoretic concepts, as well as the definition of EL-shellability. Furthermore, we recall the 
definition of Coxeter groups, and the construction of the Cambrian semilattices. In
Section~\ref{sec:shellability_cambrian}, we define a labeling of the Hasse diagram of a Cambrian
semilattice and give a case-free proof that this labeling is indeed an EL-labeling for every closed
interval of this semilattice, thus proving Theorem~\ref{thm:infinite_shellability}. In 
Section~\ref{sec:applications}, we prove Theorem~\ref{thm:cambrian_topology}, by counting the falling
maximal chains with respect to our labeling and by applying \cite[Theorem~5.9]{bjorner96shellable} which
relates the number of falling maximal chains in a poset to the homotopy type of the corresponding order
complex. 
The characterization of the spherical intervals of $C_{\gamma}$ follows from 
Theorem~\ref{thm:mobius_nuclear}. 
 
\section{Preliminaries}
  \label{sec:preliminaries}
In this section, we recall the necessary definitions, which are used throughout the article. For 
further background on posets, we refer to \cite{davey02introduction} or to 
\cite{stanley01enumerative}, where in addition some background on lattices and lattice congruences 
is provided. An introduction to poset topology can be found in either 
\cite{bjorner96topological} or \cite{wachs07poset}. For more background on Coxeter groups, we refer to
\cite{bjorner05combinatorics} and \cite{humphreys90reflection}.
 
\subsection{Posets and EL-Shellability}
  \label{sec:shellability}
Let $(P,\leq_{P})$ be a finite partially ordered set (\emph{poset} for short). We say that $P$ is 
\emph{bounded} if it has a unique minimal and a unique maximal element, which we usually denote by
$\hat{0}$ and $\hat{1}$, respectively. For $x,y\in P$, we say that $y$ \emph{covers} $x$ (and write 
$x\lessdot_{P} y$) if $x\leq_{P} y$ and there is no $z\in P$ such that $x<_{P} z<_{P} y$. We denote the
set of all covering relations of $P$ by $\mathcal{E}(P)$.  

For $x,y\in P$ with $x\leq_{P} y$, we define the closed interval $[x,y]$ to be the set 
$\{z\in P\mid x\leq_{P} z\leq_{P} y\}$. Similarly, we define the open interval 
$(x,y)=\{z\in P\mid x<_{P}z<_{P}y\}$. A chain $c:x=p_0\leq_{P} p_1\leq_{P}\cdots\leq_{P} p_s=y$ is 
called \emph{maximal} if $(p_i,p_{i+1})\in \mathcal{E}(P)$ for every $0\leq i\leq s-1$. 

Let $(P,\leq_{P})$ be a bounded poset and let
$c:\hat{0}=p_0\lessdot_{P} p_1\lessdot_{P}\cdots\lessdot_{P} p_s=\hat{1}$ be a maximal chain of $P$. Given 
another poset $(\Lambda,\leq_{\Lambda})$, a map $\lambda:\mathcal{E}(P)\to\Lambda$ is called 
\emph{edge-labeling of $P$}. We denote the sequence 
$\bigl(\lambda(p_0,p_1),\lambda(p_1,p_2),\ldots,\lambda(p_{s-1},p_s)\bigr)$ of edge-labels of $c$ by 
$\lambda(c)$. The chain $c$ is called \emph{rising} (respectively \emph{falling}) if
$\lambda(c)$ is a strictly increasing (respectively weakly decreasing) sequence. 
For two words $(p_1,p_2,\ldots,p_s)$ and $(q_1,q_2,\ldots ,q_t)$ in the alphabet $\Lambda$, we write 
$(p_1,p_2,\dots,p_s)\leq_{\Lambda^*}(q_1,q_2,\dots,q_t)$ if and only if either 
\begin{align*}
	& p_i=q_i, && \hspace*{-3cm}\mbox{for}\;1\leq i\leq s\;\mbox{and}\;s\leq t,\quad\mbox{or}\\
	& p_i<_{\Lambda}q_i, && \hspace*{-3cm}\mbox{for the least}\;i\;\mbox{such that}\;p_i\neq q_i.
\end{align*}
A maximal chain $c$ of $P$ is called \emph{lexicographically first} among the maximal chains of $P$ 
if for every other maximal chain $c'$ of $P$ we have $\lambda(c)\leq_{\Lambda^*} \lambda(c')$. 
An edge-labeling of $P$ is called \emph{EL-labeling} if for every closed interval $[x,y]$ in $P$ there
exists a unique rising maximal chain which is lexicographically first among all maximal chains in
$[x,y]$. A bounded poset that admits an EL-labeling is called \emph{EL-shellable}.

Let us recall that the M\"obius function $\mu$ of $P$ is the map $\mu:P\times P\to\mathbb{Z}$ defined
recursively by
\begin{displaymath}
	\mu(x,y)=\begin{cases}
		1, & x=y\\
		-\sum_{x\leq_{P}z<_{P}y}{\mu(x,z)}, & x<_{P}y\\
		0, & \mbox{otherwise}.
	\end{cases}
\end{displaymath}

A remarkable property of EL-shellable posets is that we can compute the value of the M\"obius function
for every closed interval of $P$ from the labeling, as is stated in the following 
proposition\footnote{Actually, Proposition~5.7 in \cite{bjorner96shellable} is stated for posets admitting
a so-called \emph{CR-labeling}. EL-shellable posets are a particular instance of this class of posets, and 
for the scope of this article it is sufficient to restrict our attention to these.}.
\begin{proposition}[\cite{bjorner96shellable}*{Proposition~5.7}]
	\label{prop:mobius}
	Let $(P,\leq_{P})$ be an EL-shellable poset, and let $x,y\in P$ with $x\leq_{P} y$. Then,
	\begin{multline*}
		\mu(x,y)=\mbox{number of even length falling maximal chains in}\;[x,y]\\
			-\;\mbox{number of odd length falling maximal chains in}\;[x,y].
	\end{multline*}
\end{proposition}

\subsection{Coxeter Groups and Weak Order}
  \label{sec:weak_order}
Let $W$ be a (possibly infinite) group, which is generated by the finite set 
$S=\{s_{1},s_{2},\ldots,s_{n}\}$, where $\varepsilon\in W$ denotes the identity. Let 
$m=(m_{i,j})_{1\leq i,j\leq n}$ be a symmetric $(n\times n)$-matrix, where the entries are either 
positive integers or the formal symbol $\infty$, and which satisfies $m_{i,i}=1$ for all $1\leq i\leq n$, 
and $m_{i,j}\geq 2$ otherwise. (We use the convention that $\infty$ is formally larger than any natural 
number.) We call $W$ a \emph{Coxeter group} if its generators satisfy 
\begin{displaymath}
	(s_{i}s_{j})^{m_{i,j}}=\varepsilon,\quad\mbox{for}\;1\leq i,j\leq n.
\end{displaymath}
We interpret the case $m_{i,j}=\infty$ as stating that there is no relation between the generators $s_{i}$ 
and $s_{j}$, and call the matrix $m$ the \emph{Coxeter matrix of $W$}. The \emph{Coxeter diagram of $W$} is 
the graph $G=(V,E)$, with $V=S$ and $E=\bigl\{\{s_{i},s_{j}\}\mid m_{i,j}\geq 3\bigr\}$. In addition, an
edge $\{s_{i},s_{j}\}$ of $G$ is labeled by the value $m_{i,j}$ if and only if $m_{i,j}\geq 4$.

Since $S$ is a generating set of $W$, we can write every element $w\in W$ as a product of the elements
in $S$, and we call such a word a \emph{reduced word for $w$} if it has minimal length. More
precisely, define the \emph{word length} on $W$ (with respect to $S$) as
\begin{displaymath}
	\ell_{S}:W\to\mathbb{N},\quad 
		w\mapsto\min\{k\mid w=s_{i_{1}}s_{i_{2}}\cdots s_{i_{k}}\;
		\mbox{and}\;s_{i_{j}}\in S\;\mbox{for all}\;1\leq j\leq k\}.
\end{displaymath}
If $\ell_{S}(w)=k$, then every product of $k$ generators which yields $w$ is a reduced word for $w$. 
Define the \emph{(right) weak order of $W$} by 
\begin{displaymath}
	u\leq_{S} v\quad\mbox{if and only if}\quad \ell_{S}(v)=\ell_{S}(u)+\ell_{S}(u^{-1}v).
\end{displaymath}
The poset $(W,\leq_{S})$ is a graded meet-semilattice, the so-called \emph{weak-order semilattice of $W$}, 
and $\ell_{S}$ is its rank function. Moreover, $(W,\leq_{S})$ is \emph{finitary} meaning that every closed 
interval of $(W,\leq_{S})$ is finite. In the case where $W$ is finite, there exists a unique longest word 
$w_{o}$ of $W$, and $(W,\leq_{S})$ is a lattice.

\subsection{Coxeter-Sortable Words}
  \label{sec:coxeter_sortable}
From now on, we consider the Coxeter element $\gamma=s_{1}s_{2}\cdots s_{n}$, and define the half-infinite 
word 
\begin{displaymath}
	\gamma^{\infty}=s_{1}s_{2}\cdots s_{n}\vert s_{1}s_{2}\cdots s_{n}\vert\cdots. 
\end{displaymath}
The vertical bars in the representation of $\gamma^{\infty}$ are ``dividers'', which have no influence
on the structure of the word, but shall serve for a better readability. Clearly, every reduced word
for $w\in W$ can be considered as a subword of $\gamma^{\infty}$. Among all reduced words for $w$,
there is a unique reduced word, which is lexicographically first considered as a subword of 
$\gamma^{\infty}$. This reduced word is called the \emph{$\gamma$-sorting word of $w$}. 
\begin{example}
	Consider the Coxeter group $W=\mathfrak{S}_{5}$, generated by $S=\{s_{1},s_{2},s_{3},s_{4}\}$, where 
	$s_{i}$ corresponds to the transposition $(i,i+1)$ for all $i\in\{1,2,3,4\}$ and let
	$\gamma=s_1s_2s_3s_4$. Clearly, $s_{1}$ and $s_{4}$ commute. Hence, 
	$w_{1}=s_{1}s_{2}\vert s_{1}s_{4}$ and $w_{2}=s_{1}s_{2}s_{4}\vert s_{1}$ are reduced words
	for the same element $w\in W$. Considering $w_{1}$ and $w_{2}$ as subwords of $\gamma^{\infty}$, we 
	find that $w_{2}$ is a lexicographically smaller subword of $\gamma^{\infty}$ than $w_{1}$ is. There 
	are six other reduced words for $w$, namely
	\begin{align*}
		w_{3}=s_{1}s_{4}\vert s_{2}\vert s_{1}, && w_{4} = s_{4}\vert s_{1}s_{2}\vert s_{1}, 
		  && w_{5} = s_{4}\vert s_{2}\vert s_{1}s_{2},\\
		w_{6}=s_{2}s_{4}\vert s_{1}s_{2}, && w_{7} = s_{2}\vert s_{1}s_{4}\vert s_{2}, 
		  && w_{8} = s_{2}\vert s_{1}s_{2}s_{4}.
	\end{align*}
	It is easy to see that among these $w_{2}$ is the lexicographically first subword of 
	$\gamma^{\infty}$, and hence $w_{2}$ is the $\gamma$-sorting word of $w$.
\end{example}

In the following, we consider only $\gamma$-sorting words, and write
\begin{equation}
  \label{eq:presentation}
	w=s_1^{\delta_{1,1}}s_2^{\delta_{1,2}}\cdots s_n^{\delta_{1,n}}\,\vert\,
	s_1^{\delta_{2,1}}s_2^{\delta_{2,2}}\cdots s_n^{\delta_{2,n}}\,\vert\,
	\cdots\,\vert\,s_1^{\delta_{l,1}}s_2^{\delta_{l,2}}\cdots s_n^{\delta_{l,n}},
\end{equation}
where $\delta_{i,j}\in\{0,1\}$ for $1\leq i\leq l$ and $1\leq j\leq n$. For each 
$i\in\{1,2,\ldots,l\}$, we say that 
\begin{displaymath}
	b_{i}=\{s_{j}\mid \delta_{i,j}=1\}\subseteq S
\end{displaymath}
is the \emph{$i$-th block of $w$}. We consider the blocks of $w$ sometimes as sets and sometimes as
subwords of $\gamma$, depending on how much structure we need. We say that $w$ is 
\emph{$\gamma$-sortable} if and only if $b_{1}\supseteq b_{2}\supseteq\cdots\supseteq b_{l}$.
\begin{example}
	Let us continue the previous example. We have seen that $w_{2}=s_{1}s_{2}s_{4}\vert s_{1}$ is a 
	$\gamma$-sorting word in $W$, and $b_{1}=\{s_{1},s_{2},s_{4}\}$, and $b_{2}=\{s_{1}\}$. Since 
	$b_{2}\subseteq b_{1}$, we see that $w_{2}$ is indeed $\gamma$-sortable.
\end{example}

The $\gamma$-sortable words of $W$ are characterized by a recursive property which we will describe next. A
generator $s\in S$ is called \emph{initial in $\gamma$} if it is the first letter in some reduced word for
$\gamma$. For some subset $J\subseteq S$, we denote by $W_{J}$ the parabolic subgroup of $W$ generated by
the set $J$, and for $s\in S$ we write $\langle s\rangle=S\setminus\{s\}$. For $w\in W$, and $J\subseteq S$, 
denote by $w_{J}$ the restriction of $w$ to the parabolic subgroup $W_{J}$.

\begin{proposition}[\cite{reading11sortable}*{Proposition~2.29}]
  \label{prop:recursion}
	Let $W$ be a Coxeter group, $\gamma$ a Coxeter element and let $s$ be initial in $\gamma$. Then an 
	element $w\in W$ is $\gamma$-sortable if and only if
	\begin{enumerate}
		\item[(i)] $s\leq_{S}w$ and $sw$ is $s\gamma s$-sortable,\quad or
		\item[(ii)] $s\not\leq_{S}w$ and $w$ is an $s\gamma$-sortable word of $W_{\langle s\rangle}$.
	\end{enumerate}
\end{proposition}

\begin{remark}
  \label{rem:dependence}
	The property of being $\gamma$-sortable does not depend on the choice of a reduced word for $\gamma$,
	see \cite{reading11sortable}*{Section~2.7}. 
	For $w\in W$, let $w_{1}$ and $w_{2}$ be the $\gamma$-sorting words of $w$
	with respect to two different reduced words $\gamma_{1}$ and $\gamma_{2}$ for $\gamma$. 
	Since $\gamma_{1}$ and $\gamma_{2}$ differ only in commutations of letters, it is clear 
	that $w_{1}$ and $w_{2}$ differ also only in commutations of letters, with no commutations across 
	dividers. Hence, the $i$-th block of $w_{1}$, considered as a subset of $S$, is equal to the $i$-th 
	block of $w_{2}$, considered as a subset of $S$. However, the $i$-th block of $w_{1}$, considered as 
	a subword of $\gamma_{1}$ , is different from the $i$-th block of $w_{2}$, considered as a subword 
	of $\gamma_{2}$. 
\end{remark}

\subsection{Cambrian Semilattices}
  \label{sec:cambrian}
In \cite{reading11sortable}*{Section~7} the \emph{Cambrian semilattice $C_{\gamma}$} was defined as the 
sub-semilattice of the weak order on $W$ consisting of all $\gamma$-sortable elements. That $C_{\gamma}$ is
well-defined follows from the following theorem.

\begin{theorem}[\cite{reading11sortable}*{Theorem~7.1}]
  \label{thm:joins_meets}
	Let $A$ be a collection of $\gamma$-sortable elements of $W$. If $A$ is nonempty, then $\bigwedge A$
	is $\gamma$-sortable. If $A$ has an upper bound, then $\bigvee A$ is $\gamma$-sortable. 
\end{theorem}

It turns out that $C_{\gamma}$ is not only a sub-semilattice of the weak order, but also a quotient 
semilattice. The key role in the proof of this property plays the projection $\pi_{\downarrow}^{\gamma}$ which
maps every word $w\in W$ to the unique largest $\gamma$-sortable element below $w$. More precisely if $s$
is initial in $\gamma$, then define
\begin{equation}
  \label{eq:projection}
	\pi_{\downarrow}^{\gamma}(w)=
	  \begin{cases}
	    s\pi_{\downarrow}^{s\gamma s}(sw), & \mbox{if}\;s\leq_{S}w\\
	    \pi_{\downarrow}^{s\gamma}(w_{\langle s\rangle}), & \mbox{if}\;s\not\leq_{S}w,
	  \end{cases}
\end{equation}
and set $\pi_{\downarrow}^{\gamma}(\varepsilon)=\varepsilon$, see \cite{reading11sortable}*{Section~6}. 
The most important properties of this map are stated in the following theorems.

\begin{theorem}[\cite{reading11sortable}*{Theorem~6.1}]
  \label{thm:projection_homomorphism}
	The map $\pi_{\downarrow}^{\gamma}$ is order-preserving. 
\end{theorem}

\begin{theorem}[\cite{reading11sortable}*{Theorem~7.3}]
	For some subset $A\subseteq W$, if $A$ is nonempty, then 
	$\bigwedge\pi_{\downarrow}^{\gamma}(A)=\pi_{\downarrow}^{\gamma}\bigl(\bigwedge A\bigr)$ and if $A$ 
	has an upper bound, then 
	$\bigvee\pi_{\downarrow}^{\gamma}(A)=\pi_{\downarrow}^{\gamma}\bigl(\bigvee A\bigr)$.
\end{theorem}

Hence, $\pi_{\downarrow}^{\gamma}$ is a semilattice homomorphism from the weak order on $W$ to $C_{\gamma}$, 
and $C_{\gamma}$ can be considered as the quotient semilattice of the weak order modulo the semilattice 
congruence $\theta_{\gamma}$ induced by the fibers of $\pi_{\downarrow}^{\gamma}$. This semilattice 
congruence is called \emph{Cambrian congruence}. Since the lack of a maximal element is the only obstruction
for the weak order to be a lattice, it follows immediately that the restriction of 
$\pi_{\downarrow}^{\gamma}$ (and hence $\theta_{\gamma}$) to closed intervals of the weak order yields a 
lattice homomorphism (and hence a lattice congruence). Figure~\ref{fig:cambrian_congruence} shows the
Hasse diagram of the weak order on the Coxeter group $A_{3}$ and the congruence classes of 
$\theta_{\gamma}$ for $\gamma=s_{1}s_{2}s_{3}$. 

\begin{figure}
	\centering
	\cambCongA{2.25}{2.25}
	\caption{The Cambrian congruence on the weak-order lattice on $A_{3}$ induced by the Coxeter
	element $s_{1}s_{2}s_{3}$. The non-singleton congruence classes are highlighted.}
	\label{fig:cambrian_congruence}
\end{figure}

In the remainder of this article, we switch frequently between the weak-order semilattice on $W$ and the
Cambrian semilattice $C_{\gamma}$. In order to point out properly which semilattice we consider, we denote
the order relation of the weak-order semilattice by $\leq_{S}$, and the order relation of $C_{\gamma}$ by
$\leq_{\gamma}$. Analogously, we denote a closed (respectively open) interval in the weak-order semilattice 
by $[u,v]_{S}$ (respectively $(u,v)_S$), and a closed (respectively open) interval in $C_{\gamma}$
by $[u,v]_{\gamma}$ (respectively $(u,v)_{\gamma}$). 

\section{EL-Shellability of the Closed Intervals in $C_{\gamma}$}
  \label{sec:shellability_cambrian}
In this section, we define an edge-labeling of $C_{\gamma}$, discuss some of its properties and eventually
prove Theorem~\ref{thm:infinite_shellability}. 

\subsection{The Labeling}
  \label{sec:labeling}
Define for every $w\in W$ the set of positions of the $\gamma$-sorting word of $w$ as 
\begin{displaymath}
	\alpha_{\gamma}(w)=\bigl\{(i-1)\cdot n+j\mid\delta_{i,j}=1\bigr\}\subseteq\mathbb{N},
\end{displaymath}
where the $\delta_{i,j}$'s are the exponents from \eqref{eq:presentation}.
In view of Remark~\ref{rem:dependence}, we notice that the set of positions of $w$ depends not only on the 
choice of the Coxeter element $\gamma$, but also on the choice of the reduced word for $\gamma$. 

\begin{example}
	Let $W=\mathfrak{S}_{4}$, $\gamma=s_1s_2s_3$  and consider 
	$u=s_{1}s_{2}s_{3}\vert s_{2}$, and $v=s_{2}s_{3}\vert s_{2}\vert s_{1}$. 
	Then, $\alpha_{\gamma}(u)=\{1,2,3,5\}$, and $\alpha_{\gamma}(v)=\{2,3,5,7\}$, where $u\in C_{\gamma}$, while
	$v\notin C_{\gamma}$. 
\end{example}
It is not hard to see that an element $w\in W$ lies in $C_{\gamma}$ 
if and only if for all $i>n$ the following holds: if $i\in\alpha_{\gamma}(w)$, then $i-n\in\alpha_{\gamma}(w)$. In the
previous example, we see that $\alpha_{\gamma}(u)$ contains both $5$ and $2$, while $\alpha_{\gamma}(v)$ does not
contain $7-3=4$. 

\begin{lemma}
  \label{lem:alpha}
  	Let $u,v\in W$ with $u\leq_S v$. Then $\alpha_{\gamma}(u)$ is a subset of $\alpha_{\gamma}(v)$.
\end{lemma}
\begin{proof}
    The $\gamma$-sorting word of an element $w\in W$ is a reduced word for $w$. 
	Thus, it follows immediately from the definition of the weak order that any letter appearing in 
	the $\gamma$-sorting word of $u$ has to appear also in the $\gamma$-sorting word of every element
	that is greater than $w$ in the weak order. 
	Thus, if $u,v\in C_{\gamma}$ with $u\leq_{\gamma}v$, then $\alpha_{\gamma}(u)\subseteq \alpha_{\gamma}(v)$. 
\end{proof}

Denote by $\mathcal{E}(C_{\gamma})$ the set of covering relations of $C_{\gamma}$, and define 
an edge-labeling of $C_{\gamma}$ by
\begin{equation}
  \label{eq:labeling}
	\lambda_{\gamma}:\mathcal{E}(C_{\gamma})\to\mathbb{N},
	  \quad (u,v)\mapsto\min\{i\mid i\in\alpha_{\gamma}(v)\smallsetminus\alpha_{\gamma}(u)\}.
\end{equation}
Figures~\ref{fig:camb_a3} and \ref{fig:camb_b3} show the Hasse diagrams of a Cambrian lattice 
$C_{\gamma}$ of the Coxeter groups $A_{3}$ and $B_{3}$ respectively, together with the labels defined by
the map $\lambda_{\gamma}$. 

\begin{figure}
	\centering
	\cambAThree{2}{2}
	\caption{An $A_{3}$-Cambrian lattice with the labeling as defined in \eqref{eq:labeling}.}
	\label{fig:camb_a3}
\end{figure}

\begin{figure}
	\centering
	\cambBThree{1.25}{1.5}
	\caption{A $B_{3}$-Cambrian lattice, with the labeling as defined in \eqref{eq:labeling}.}
	\label{fig:camb_b3}
\end{figure}

\subsection{Properties of the Labeling}
  \label{sec:properties}
Again in view of Remark~\ref{rem:dependence}, we notice that the definition of $\lambda_{\gamma}$ depends on 
a specific reduced word for $\gamma$. The following lemma shows that the structural properties of 
$\lambda_{\gamma}$ required for the purpose of this article are, however, independent of the choice of 
reduced word for $\gamma$.

\begin{lemma}
  \label{lem:well_defined}
	Let $\gamma\in W$ be a Coxeter element, and let $u,v\in C_{\gamma}$ with $u\leq_{\gamma}v$. The 
	number of maximal falling and rising chains in $[u,v]_{\gamma}$ does not depend on the choice of a 
	reduced word for $\gamma$.
\end{lemma}
\begin{proof}
	Say that $w_{1}$ and $w_{2}$ are two reduced words for $\gamma$. Without loss of generality
	we can assume that $w_{2}$ is obtained from $w_{1}$ by exchanging two commuting letters $s,t\in S$,
	and we may assume that $s$ appears before $t$ in $w_{1}$. We write $\lambda_{w_{1}}$ and
	$\lambda_{w_{2}}$ to indicate which reduced word for $\gamma$ we consider, and say that $s$ is the 
	$k$-th letter of $w_{1}$ (thus $t$ is the $(k+1)$-st letter of $w_{1}$, and vice versa for $w_{2}$).
	Let 
	$c:u=x_{0}\lessdot_{\gamma}x_{1}\lessdot_{\gamma}\cdots\lessdot_{\gamma}x_{t}=v$ be a rising chain 
	with respect to the labeling $w_1$.
	
	(1) Suppose that there is a minimal index $j$ such that $\lambda_{w_{1}}(x_{j-1},x_{j})=k+(l-1)n$ for 
	some $l\geq 1$.\quad  Thus, 
	$x_{j}$ is obtained from $x_{j-1}$ by inserting the letter $s$ into the $l$-th block of $x_{j-1}$
	(and possibly inserting more letters into later blocks.) Since $c$ is rising, we know that 
	$\lambda_{w_{1}}(x_{j-2},x_{j-1})<k+(l-1)n<\lambda_{w_{1}}(x_{j},x_{j+1})$. Moreover, since $s$ 
	appears before $t$ in $w_{1}$, and since $j$ is minimal, we conclude 
	$\lambda_{w_{1}}(x_{i},x_{i+1})=\lambda_{w_{2}}(x_{i},x_{i+1})$, for every $i\leq j-2$. 
	Since $w_{2}$ is obtained from 
	$w_{1}$ by exchanging $s$ and $t$, we have 
	$\lambda_{w_{2}}(x_{j-1},x_{j})=\lambda_{w_{1}}(x_{j-1},x_{j})+1=(k+1)+(l-1)n$. 
	
	(1a) If $\lambda_{w_{1}}(x_{j},x_{j+1})>(k+1)+(l-1)n$, then $x_{j+1}$ is obtained from $x_{j}$ either 
	by inserting a letter which appears after $t$ in $w_{1}$ into the $l$-th block of $x_{j}$, or by 
	inserting some letter into the $l'$-th block of $x_{j}$, where $l'>l$ (and possibly inserting more
	letters into later blocks). In both cases, we have 
	$\lambda_{w_{2}}(x_{j},x_{j+1})>(k+1)+(l-1)n$. 
	
	\begin{figure}
		\centering
		\begin{tikzpicture}\scriptsize
			\def\x{1};
			\def\y{1};
			\draw(2*\x,1*\y) node[circle,scale=.5,draw,label=left:$u$](nu){};
			\draw(2*\x,7*\y) node[circle,scale=.5,draw,label=right:$v$](nv){};
			\draw(2*\x,3*\y) node[circle,scale=.5,draw,label=left:$x_{j-1}$](nx1){};
			\draw(1*\x,4*\y) node[circle,scale=.5,draw,label=left:$x_{j}$](nx2){};
			\draw(2*\x,5*\y) node[circle,scale=.5,draw,label=right:$x_{j+1}$](nx3){};
			\draw(3*\x,4*\y) node[circle,scale=.5,draw,label=right:$x'$](nx4){};
			\draw(nx1) -- (nx2);
			\draw(nx1) -- (nx4);
			\draw(nx2) -- (nx3);
			\draw(nx4) -- (nx3);
			\draw[rounded corners](nu) -- (1.67*\x,1.67*\y) -- (2.33*\y,2.33*\y) -- (nx1);
			\draw[rounded corners](nx3) -- (1.67*\x,5.67*\y) -- (2.33*\x,6.33*\y) -- (nv);
			\draw(3*\x,3.05*\y) node{$\leftarrow s,t\notin b_{l}$};
			\draw(4.5*\x,4*\y) node{$\leftarrow s\notin b_{l},t\in b_{l}$};
			\draw(3.5*\x,5.05*\y) node{$\leftarrow s,t\in b_{l}$};
			\draw(-.5*\x,4.05*\y) node{$s\in b_{l},t\notin b_{l}\rightarrow$};
		\end{tikzpicture}
		\caption{Illustrating the case $\lambda_{w_{1}}(x_{j},x_{j+1})=(k+1)+(l-1)n$ in the proof of
		  Lemma~\ref{lem:well_defined}.}
		\label{fig:commuting}
	\end{figure}
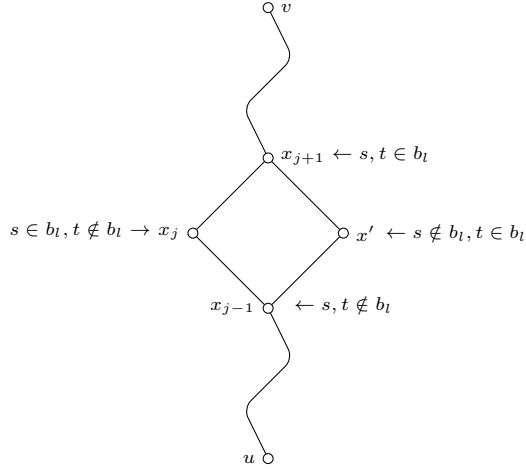
	
	(1b) If $\lambda_{w_{1}}(x_{j},x_{j+1})=(k+1)+(l-1)n$, then $x_{j+1}$ is obtained from $x_{j}$ by 
	inserting the letter $t$ into the $l$-th block of $x_{j}$ (and possibly inserting more letters into 
	later blocks), which implies $\lambda_{w_{2}}(x_{j},x_{j+1})=k+(l-1)n$. Hence, $c$ is not rising 
	with respect to $\lambda_{w_{2}}$. However, $x_{j+1}$ is obtained from $x_{j-1}$ by inserting the 
	letters $s$ and $t$ into the $l$-th block of $x_{j-1}$ (and possibly inserting more letters into 
	later blocks). Since $s$ and $t$ commute it does not matter which letter is inserted first. (Note 
	that we need here that the $\gamma$-sortability of $x_{j+1}$ does not depend on a reduced word for 
	$\gamma$, see \cite{reading11sortable}*{Lemma~6.6}.) This means in particular that the word $x'$ 
	obtained from $x_{j-1}$ by inserting the letter $t$ into the $l$-th block of $x_{j-1}$ (and possibly 
	inserting more letters into later blocks) is $\gamma$-sortable, and we have 
	$x_{j-1}\lessdot_{\gamma}x'\lessdot_{\gamma}x_{j+1}$. It follows that
	$\lambda_{w_{1}}(x_{j-1},x')=(k+1)+(l-1)n$ and $\lambda_{w_{1}}(x',x_{j+1})=k+(l-1)n$. This implies
	$\lambda_{w_{2}}(x_{j-1},x')=k+(l-1)n$ and $\lambda_{w_{2}}(x',x_{j+1})=(k+1)+(l-1)n$. Thus,
	the chain $c':u=x_{0}\lessdot_{\gamma}x_{1}\lessdot_{\gamma}\cdots\lessdot_{\gamma}x_{j-1}
	\lessdot_{\gamma}x'\lessdot_{\gamma}x_{j+1}$ is rising 
	with respect to $\lambda_{w_{2}}$ but not rising with respect to $\lambda_{w_{1}}$. (See 
	Figure~\ref{fig:commuting} for an illustration.) 
	
	We repeat the same procedure if there exists another index $j'>j$ such that 
	$\lambda_{w_{1}}(x_{j'-1},x_{j'})=k+(l'-1)n$, for some $l'>l$.
	
	(2) Suppose that for every $l\geq 1$, no label of the form $k+(l-1)n$ is present in 
	$\lambda_{w_{1}}(c)$, and there is a minimal index $j$ such that 
	$\lambda_{w_{1}}(x_{j-1},x_{j})=(k+1)+(l-1)n$.\quad By assumption and since $c$ is rising, we notice 
	that $\lambda_{w_{1}}(x_{j-2},x_{j-1})\leq k-1+(l-1)n$. Since $j$ is minimal, we conclude that
	$\lambda_{w_{2}}(x_{j-2},x_{j-1})=\lambda_{w_{1}}(x_{j-2},x_{j-1})$, and we have
	$\lambda_{w_{2}}(x_{j-1},x_{j})=\lambda_{w_{1}}(x_{j-1},x_{j})-1$. Thus, $c$ is still rising with 
	respect $\lambda_{w_{2}}$. We argue similarly if there exists another index $j'>j$ such that 
	$\lambda_{w_{1}}(x_{j'-1},x_{j'})=(k+1)+(l'-1)n$, for some $l'>l$.
	
	(3) Suppose that for every $l\geq 1$, no label of the form $k+(l-1)n$ or $(k+1)+(l-1)n$ is present 
	in $\lambda_{w_{1}}(c)$. Then, $\lambda_{w_{2}}(c)=\lambda_{w_{1}}(c)$.

	\smallskip
	
	The statement for falling chains can be shown analogously.
\end{proof}

Whenever we use an initial letter $s$ of $\gamma$ in the remainder of this article, we consider 
$\lambda_{\gamma}$ with respect to a fixed reduced word for $\gamma$ which has $s$ as its first letter. The 
previous lemma implies that this can be done without loss of generality.

\begin{lemma}
  \label{lem:i_0}
	Let $C_{\gamma}$ be a Cambrian semilattice, and let $u,v\in C_{\gamma}$ such that $u\leq_{\gamma}v$. 
	Let $i_0=\min\{i\mid i\in\alpha_{\gamma}(v)\smallsetminus\alpha_{\gamma}(u)\}$. Then 
	the following hold.
	\begin{enumerate}
		\item[(i)] The label $i_0$ appears in every maximal chain of the interval $[u,v]_{\gamma}$. 
		\item[(ii)] The labels of a maximal chain in $[u,v]_{\gamma}$ are distinct.
	\end{enumerate}
\end{lemma}
\begin{proof}
	(i) Suppose that this is not the case. Then there exists a maximal chain 
	$u=u_0\lessdot_{\gamma}u_1\lessdot_{\gamma}\cdots\lessdot_{\gamma}u_{k-1}\lessdot_{\gamma}u_k=v$
	with $\lambda_{\gamma}(u_i,u_{i+1})\neq i_0$ for every $i\in\{0,1,\dots, k-1\}$. Hence, $i_0\in\alpha_{\gamma}(u)$ 
	if and only if $i_0\in\alpha_{\gamma}(v)$, which contradicts the definition of $i_0$. 

	(ii) Let $u=c_{0}\lessdot_{\gamma}c_{1}\lessdot_{\gamma}\cdots\lessdot_{\gamma}c_{m}=v$ be a
	maximal chain in $[u,v]_{\gamma}$. Assume that there are $i,j\in\{1,2,\ldots,m\}$ with $i<j$
	such that $\lambda_{\gamma}(c_{i},c_{i+1})=k=\lambda_{\gamma}(c_{j},c_{j+1})$. By definition, $k\in\alpha_{\gamma}(c_{i+1})$,
	and $k\notin\alpha_{\gamma}(c_{j})$. Since $c_{i+1}\leq_{S} c_{j}$, we can conclude from
	Lemma~\ref{lem:alpha} that $\alpha_{\gamma}(c_{i+1})\subseteq\alpha_{\gamma}(c_{j})$, which yields a contradiction.
\end{proof}

The $\gamma$-sortable words of $W$ are defined recursively as described in
Proposition~\ref{prop:recursion}. Thus we need to investigate how our labeling behaves with respect to this 
recursion.

\begin{lemma}
   \label{lem:labeling_recursion}
	Let $W$ be a Coxeter group and let $\gamma\in W$ be a Coxeter element. For $u,v\in C_{\gamma}$ with
	$u\lessdot_{\gamma}v$ and for $s\in S$ initial in $\gamma$, we have
	\begin{displaymath}
		\lambda_{\gamma}(u,v)=
		  \begin{cases}
		  	  1, & \mbox{if}\;s\not\leq_{S}u\;\mbox{and}\;s\leq_{S}v,\\
		  	  \lambda_{s\gamma s}(su,sv)+1, & \mbox{if}\;s\leq_{S}u,\\
		  	  \lambda_{s\gamma}(u_{\langle s\rangle},v_{\langle s\rangle})+k, & \mbox{if}\;
		  	    s\not\leq_{S}v\;\mbox{and the first position where}\;u\;\mbox{and}\;v\\
		  	    & \mbox{differ is in their}\;k\mbox{-th block}.
		  \end{cases}
	\end{displaymath}
\end{lemma}
\begin{proof}
	Let first $s\not\leq_{S}u$ and $s\leq_{S}v$. By definition of the weak order, $s$ does not occur in 
	the first position of any reduced word for $u$, in particular it does not occur in the first position 
	of the $\gamma$-sorting word of $u$. Hence, $1\notin\alpha_{\gamma}(u)$. Since $s$ is initial in $\gamma$, it 
	does occur in the first position of the $\gamma$-sorting word of $v$, and hence $1\in\alpha_{\gamma}(v)$. By
	definition this implies $\lambda_{\gamma}(u,v)=1$.	
	
	Let now $s\leq_{S}u$. Then, $s\leq_{S}v$, and with Proposition~\ref{prop:recursion}, we find that
	$su$ and $sv$ are $s\gamma s$-sortable. It follows from \cite{reading11sortable}*{Proposition~2.18},
	Proposition~\ref{prop:recursion} and the definition of the weak order that 
	$su\lessdot_{s\gamma s}sv$. Say $\lambda_{s\gamma s}(su,sv)=k$. By construction, the 
	$s\gamma s$-sorting word of $su$ is precisely the subword of $u$ starting at the second position. 
	Thus, the $s\gamma s$-sorting word of $su$ is the leftmost subword of $\gamma^{\infty}$ where the 
	first position is empty, and likewise for $sv$. If the first position of $(s\gamma s)^{\infty}$ 
	where $su$ and $sv$ differ is $k$, then the first position of $\gamma^{\infty}$ where $u$ and $v$ 
	differ is $k+1$. Hence, $\lambda_{\gamma}(u,v)=\lambda_{s\gamma s}(su,sv)+1$.
	
	Finally, let $s\not\leq_{S}v$. Then, $s\not\leq_{S}u$, and with Proposition~\ref{prop:recursion}, we
	find that $u_{\langle s\rangle}$ and $v_{\langle s\rangle}$ are $s\gamma$-sortable words of the
	parabolic subgroup $W_{\langle s\rangle}$ of $W$, and the Cambrian lattice $C_{s\gamma}$ is an order 
	ideal in $C_{\gamma}$. Say that the first position filled in $v_{\langle s\rangle}$ but not in 
	$u_{\langle s\rangle}$ is in the $k$-th block of $v_{\langle s\rangle}$. Considering 
	$u_{\langle s\rangle}$ and $v_{\langle s\rangle}$ as subwords of $\gamma^{\infty}$ adds the letter 
	$s$ with exponent $0$ to each block of $u_{\langle s\rangle}$ and $v_{\langle s\rangle}$. Since the 
	first difference of $u_{\langle s\rangle}$ and $v_{\langle s\rangle}$ is in the $k$-th block, the 
	first difference of $u$ and $v$ is still in the $k$-th block, but each block has an additional first 
	letter. Hence $\lambda_{\gamma}(u,v)=\lambda_{s\gamma}(u_{\langle s\rangle},v_{\langle s\rangle})+k$.
\end{proof}

\begin{example}
	Let $W=B_{3}$ generated by $S=\{s_{1},s_{2},s_{3}\}$ satisfying
	$(s_{1}s_{2})^{3}=(s_{2}s_{3})^{4}=(s_{1}s_{3})^{2}=\varepsilon$ and 
	$s_{1}^{2}=s_{2}^{2}=s_{3}^{2}=\varepsilon$, and let $\gamma=s_{1}s_{2}s_{3}$ be a Coxeter element of 
	$B_{3}$. 
	
	Consider $u_{1}=s_{2}s_{3}\vert s_{2}s_{3}$ and 
	$v_{1}=s_{1}s_{2}s_{3}\vert s_{1}s_{2}s_{3}\vert s_{1}s_{2}s_{3}$. With the definition of our 
	labeling follows $\lambda_{\gamma}(u_{1},v_{1})=1$ immediately. 
	
	Let now $u_{2}=s_{1}s_{2}s_{3}\vert s_{1}s_{2}$ and 
	$v_{2}=s_{1}s_{2}s_{3}\vert s_{1}s_{2}s_{3}$. Then, $s_{1}u_{2}=s_{2}s_{3}s_{1}\vert s_{2}$ and
	$s_{1}v_{2}=s_{2}s_{3}s_{1}\vert s_{2}s_{3}$ considered as $s_{1}\gamma s_{1}$-sorting words. We have
	\begin{displaymath}
		\lambda_{s_{1}\gamma s_{1}}(s_{1}u_{2},s_{1}v_{2})=5,\quad\mbox{and}\quad
		\lambda_{\gamma}(u_{2},v_{2})=6.
	\end{displaymath}
	
	Finally, let $u_{3}=s_{2}s_{3}\vert s_{2}$ and $v_{3}=s_{2}s_{3}\vert s_{2}s_{3}$. The
	$(s_{1}\gamma)^{\infty}$-sorting words of $(u_{3})_{\langle s_{1}\rangle}$ and 
	$(v_{3})_{\langle s_{1}\rangle}$ written as in
	\eqref{eq:presentation} are 
	\begin{displaymath}
		(u_{3})_{\langle s_{1}\rangle}=s_{2}^{1}s_{3}^{1}\vert s_{2}^{1}s_{3}^{0},
		  \quad\mbox{and}\quad 
		(v_{3})_{\langle s_{1}\rangle}=s_{2}^{1}s_{3}^{1}\vert s_{2}^{1}s_{3}^{1}.
	\end{displaymath}
	The corresponding $\gamma$-sorting words of $u_{3}$ and $v_{3}$ are
	\begin{displaymath}
		u_{3}=s_{1}^{0}s_{2}^{1}s_{3}^{1}\vert s_{1}^{0}s_{2}^{1}s_{3}^{0},\quad\mbox{and}\quad 
		v_{3}=s_{1}^{0}s_{2}^{1}s_{3}^{1}\vert s_{1}^{0}s_{2}^{1}s_{3}^{1}.
	\end{displaymath}
	Hence, 
	$\lambda_{s_{1}\gamma}\bigl((u_{3})_{\langle s_{1}\rangle},(v_{3})_{\langle s_{1}\rangle}\bigr)=4$
	and $\lambda_{\gamma}(u_{3},v_{3})=6$.
\end{example}

\subsection{Proof of Theorem~\ref{thm:infinite_shellability}}
  \label{sec:proof}
We will prove Theorem~\ref{thm:infinite_shellability} by showing that the map $\lambda_{\gamma}$ defined in 
\eqref{eq:labeling} is an EL-labeling for every closed interval in $C_{\gamma}$. In particular we show the 
following.

\begin{theorem}
  \label{thm:el_labeling}
	Let $u,v\in C_{\gamma}$ with $u\leq_{\gamma}v$. Then the map $\lambda_{\gamma}$ defined in 
	\eqref{eq:labeling} is an EL-labeling for $[u,v]_{\gamma}$.
\end{theorem}

We notice in view of Lemma~\ref{lem:well_defined} that the statement of Theorem~\ref{thm:el_labeling} does
not depend on a reduced word for $\gamma$, even though our labeling does.

For the proof of Theorem~\ref{thm:el_labeling}, we need one more technical lemma. This lemma uses many of the
deep results on Cambrian semilattices developed in \cite{reading11sortable}, and needs the following 
alternative characterization of the (right) weak order on $W$. Let $T=\{wsw^{-1}\mid w\in W, s\in S\}$, and 
define for $w\in W$, the \emph{(left) inversion set of $w$} as 
$\mbox{inv}(w)=\{t\in T\mid\ell_{S}(tw)\leq\ell_{S}(w)\}$. It is the statement of 
\cite{bjorner05combinatorics}*{Proposition~3.1.3} that $u\leq_{S}v$ if and only if 
$\mbox{inv}(u)\subseteq \mbox{inv}(v)$. Thus, every $w\in W$ is uniquely determined by its inversion set, 
and for $J\subseteq S$ the map $w\mapsto w_{J}$ is defined by the property that 
$\mbox{inv}(w_{J})=\mbox{inv}(w)\cap W_{J}$, see \cite[Section~2.4]{reading11sortable}.

\begin{lemma}
  \label{lem:joincov}
	Let $u,v\in C_{\gamma}$ with $u\leq_{\gamma}v$ and let $s$ be initial in $\gamma$. If
	$s\not\leq_{\gamma}u$ and $s\leq_{\gamma}v$, then the join $s\vee_{\gamma}u$ covers $u$ in 
	$C_{\gamma}$.
\end{lemma}
\begin{proof}
	First of all, since $s\leq_{\gamma}v$ and $u\leq_{\gamma}v$, we conclude from 
	Theorem~\ref{thm:joins_meets} that $s\vee_{\gamma} u$ exists, and set $z=s\vee_{\gamma}u$. 
	By assumption, we have $u=\pi_{\downarrow}^{s\gamma}(u_{\langle s\rangle})\in W_{\langle s\rangle}$, 
	and Proposition~\ref{prop:recursion} implies $u=u_{\langle s\rangle}$.  
	We deduce from \cite{reading11sortable}*{Lemma~2.23} that $\mbox{cov}(z)=\{s\}\cup\mbox{cov}(u)$.
	Therefore $s$ is a cover reflection of $z$, thus it follows from 
	\cite{reading11sortable}*{Proposition~5.4~(i)} that $z=s\vee_{\gamma}z_{\langle s\rangle}$, and
	\cite{reading11sortable}*{Proposition~5.4~(ii)} implies that 
	$\mbox{cov}(z)=\{s\}\cup\mbox{cov}(z_{\langle s\rangle})$. Hence, 
	$\mbox{cov}(u)=\mbox{cov}(z_{\langle s\rangle})$, and \cite{reading11sortable}*{Theorem~8.9~(iv)} 
	implies $u=z_{\langle s\rangle}$. (The required fact that $z_{\langle s\rangle}$ is $\gamma$-sortable
	follows from \cite{reading11sortable}*{Propositions~3.13 and 6.10}.)
	
	On the other hand, it follows from the definition of a cover reflection that there exists an element 
	$z'=sz\in W$ with $z'\lessdot_{S}z$. In view of \cite{bjorner05combinatorics}*{Proposition~3.1.3}, 
	we conclude that $\mbox{inv}(z')\subseteq \mbox{inv}(z)$. Hence, we have $\mbox{inv}(z'_{\langle s\rangle})=
	\mbox{inv}(z')\cap W_{\langle s\rangle}\subseteq \mbox{inv}(z)\cap W_{\langle s\rangle}=\mbox{inv}(z_{\langle s\rangle})$,
	which implies that $z'_{\langle s\rangle}\leq_{S}z_{\langle s\rangle}$.
	Furthermore we have $\mbox{inv}(z')=\mbox{inv}(z)\setminus\{s\}$, and since $\mbox{inv}(s)=\{s\}$, 
	Proposition~3.1.3 in \cite{bjorner05combinatorics} implies $s\not\leq_{S} z'$. Hence, 
	by definition of $\pi_{\downarrow}^{\gamma}$, see \eqref{eq:projection}, we have 
	$\pi_{\downarrow}^{\gamma}(z')=
	\pi_{\downarrow}^{s\gamma}(z'_{\langle s\rangle})\in W_{\langle s\rangle}$, and 
	$\pi_{\downarrow}^{\gamma}(z')\lessdot_{\gamma}z$. Since 
	$\pi_{\downarrow}^{s\gamma}$ is order-preserving (see Theorem~\ref{thm:projection_homomorphism}), we 
	conclude from $z'_{\langle s\rangle}\leq_{S}z_{\langle s\rangle}$ that
	$\pi_{\downarrow}^{s\gamma}(z'_{\langle s\rangle})\leq_{S}
	\pi_{\downarrow}^{s\gamma}(z_{\langle s\rangle})$. Hence, 
	\begin{displaymath}
		\pi_{\downarrow}^{\gamma}(z')=\pi_{\downarrow}^{s\gamma}(z'_{\langle s\rangle})\leq_{S}
		  \pi_{\downarrow}^{s\gamma}(z_{\langle s\rangle})=\pi_{\downarrow}^{s\gamma}(u)
		    =\pi_{\downarrow}^{s\gamma}(u_{\langle s\rangle})
		      =\pi_{\downarrow}^{\gamma}(u)=u.
	\end{displaymath}
	Since $\pi_{\downarrow}^{\gamma}(z')\lessdot_{\gamma}z$ and $u<_{\gamma}z$, the previous implies
	$u=\pi_{\downarrow}^{\gamma}(z')$ and thus $u\lessdot_{\gamma}z$.
\end{proof}

\begin{proof}[Proof of Theorem~\ref{thm:el_labeling}]
	Let $[u,v]_{\gamma}$ be a closed interval of $C_{\gamma}$. Since the weak order on $W$ is finitary,
	it follows that $[u,v]_{\gamma}$ is a finite lattice. We show that there 
	exists a unique maximal rising chain which is the lexicographically first among all maximal chains 
	in this interval.
	
	We proceed by induction on length and rank, using the recursive structure of $\gamma$-sortable words, 
	see Proposition~\ref{prop:recursion}. We assume that $\ell_{S}(v)\geq 3$, and that $W$ is a 
	Coxeter group of rank $\geq 2$, since the result is trivial otherwise. Say that $W$ is of 
	rank $n$, and say that $\ell_{S}(v)=k$. Suppose that the induction hypothesis is true for 
	all parabolic subgroubs of $W$ having rank $<n$ and suppose that for every closed interval 
	$[u',v']_{\gamma}$ of $C_{\gamma}$ with $\ell_{S}(v')<k$, there exists a unique rising maximal chain 
	from $u'$ to $v'$ which is lexicographically first among all maximal chains in $[u',v']_{\gamma}$. 
	We show that there is a unique rising maximal chain in the interval $[u,v]_{\gamma}$ wich is 
	lexicographically first among all maximal chains in $[u,v]_{\gamma}$. For $s$ initial in $\gamma$, 
	we distinguish two cases: (1) $s\not\leq_{\gamma} v$ and (2) $s\leq_{\gamma} v$. 

	(1) Since $s\not\leq_{\gamma} v$, it follows that no element of $[u,v]_{\gamma}$ contains the 
	letter $s$ in its $\gamma$-sorting word. We consider the parabolic Coxeter group 
	$W_{\langle s\rangle}$ (generated by $S\setminus\{s\}$) and the Coxeter element 
	$s\gamma$. It follows from Proposition~\ref{prop:recursion} that the interval 
	$[u,v]_{\gamma}$ is isomorphic to the interval 
	$[u_{\langle s\rangle},v_{\langle s\rangle}]_{s\gamma}$ in $W_{\langle s\rangle}$. 
	Since the rank of $W_{\langle s\rangle}$ is $n-1<n$, by induction there exists a unique 
	maximal rising chain $c':u_{\langle s\rangle}=(x_{0})_{\langle s\rangle}
	\lessdot_{s\gamma}(x_{1})_{\langle s\rangle}\lessdot_{s\gamma}\cdots
	\lessdot_{s\gamma}(x_{t})_{\langle s\rangle}=v_{\langle s\rangle}$ which is lexicographically 
	first among all maximal chains in $[u_{\langle s\rangle},v_{\langle s\rangle}]_{s\gamma}$. Let 
	$(x_{j_{a}})_{\langle s\rangle}\lessdot_{s\gamma}(x_{j_{a}+1})_{\langle s\rangle}$ and 
	$(x_{j_{b}})_{\langle s\rangle}\lessdot_{s\gamma}(x_{j_{b}+1})_{\langle s\rangle}$ be two 
	covering relations in $c'$ with $j_{a}+1\leq j_{b}$. Say that the first block where 
	$(x_{j_{a}})_{\langle s\rangle}$ and $(x_{j_{a}+1})_{\langle s\rangle}$ differ is the $d_{a}$-th 
	block of their $s\gamma$-sorting word and say that the first block where 
	$(x_{j_{b}})_{\langle s\rangle}$ and $(x_{j_{b}+1})_{\langle s\rangle}$ differ is the $d_{b}$-th 
	block of their $s\gamma$-sorting word. Since $c'$ is rising, we conclude that $d_{a}\leq d_{b}$,
	and Lemma~\ref{lem:labeling_recursion} implies that the corresponding maximal chain
	$c:u=x_{0}\lessdot_{\gamma}x_{1}\lessdot_{\gamma}\cdots\lessdot_{\gamma} x_{t}=v$ in 
	$[u,v]_{\gamma}$ is rising. Similarly, it follows that $c$ is the unique maximal rising chain and 
	that it is lexicographically first among all maximal chains in $[u,v]_{\gamma}$.

	(2a) Suppose first that $s\leq_{\gamma} u$ as well. Then, $s$ is the first letter in the 
	$\gamma$-sorting word of every element in $[u,v]_{\gamma}$. It follows from 
	\cite[Proposition~2.18]{reading11sortable} and Proposition \ref{prop:recursion}
	that the interval $[u,v]_{\gamma}$ is isomorphic to the interval $[su,sv]_{s\gamma s}$.
	Moreover, Lemma~\ref{lem:labeling_recursion} implies that for a covering relation 
	$x\lessdot_{\gamma}y$ in $[u,v]_{\gamma}$ we have 
	$\lambda_{\gamma}(x,y)=\lambda_{s\gamma s}(sx,sy)+1$. Say that $c':su=sx_{0}
	\lessdot_{s\gamma s}sx_{1}\lessdot_{s\gamma s}\cdots\lessdot_{s\gamma s}sx_{t}=sv$ is
	the unique rising maximal chain in $[su,sv]_{s\gamma s}$. (This chain exists by induction, since 
	$\ell_{S}(sv)<\ell_{S}(v)$.) Then, the chain 
	$c:u=x_{0}\lessdot_{\gamma}x_{1}\lessdot_{\gamma}\cdots\lessdot_{\gamma}x_{t}=v$ is a 
	maximal chain in $[u,v]_{\gamma}$ and clearly rising. With Lemma~\ref{lem:labeling_recursion}, 
	we find that $c$ is the unique rising chain and every other maximal chain in $[u,v]_{\gamma}$ is 
	lexicographically larger than $c$. 

	(2b) Suppose now that $s\not\leq_{\gamma} u$. Since $s\leq_{\gamma} v$ and $u\leq_{\gamma}v$ the 
	join $u_{1}=s\vee_{\gamma} u$ exists and lies in $[u,v]_{\gamma}$. Lemma~\ref{lem:joincov} implies 
	that $u\lessdot_{\gamma} u_{1}$. Consider the interval $[u_1,v]_{\gamma}$. Then 
	$s\leq_{\gamma} u_1$ and analogously to (2a) we can find a unique maximal rising chain 
	$c':u_1=x_1\lessdot_{\gamma} x_2\lessdot_{\gamma}\cdots\lessdot_{\gamma}x_t=v$ in 
	$[u_1,v]_{\gamma}$ which is lexicographically first. Moreover, 
	$\min\{i\mid i\in\alpha_{\gamma}(v)\smallsetminus\alpha_{\gamma}(u_1)\}>1$, since 
	$s\leq_{\gamma}u_{1}\leq_{\gamma}v$. By definition of our labeling, the label $1$ cannot 
	appear as a label in any chain in the interval $[u_{1},v]_{\gamma}$. 
	On the other hand, it follows from Lemma~\ref{lem:labeling_recursion} that 
	$\lambda_{\gamma}(u,u_1)=1$. Thus, the chain $c:u=x_0\lessdot_{\gamma}x_1
	\lessdot_{\gamma}x_2\lessdot_{\gamma}\cdots\lessdot_{\gamma}x_t=v$ is maximal
	and rising in $[u,v]_{\gamma}$. Suppose that there is another element $u'$ that covers $u$ 
	in $[u,v]_{\gamma}$ such that $\lambda_{\gamma}(u,u')=1$. Then, by definition of 
	$\lambda_{\gamma}$, it follows that $s$ appears in the $\gamma$-sorting word of $u'$. In 
	particular, since $s$ is initial in $\gamma$, we deduce that $s\leq_{\gamma} u'$. Therefore 
	$u'$ is above both $s$ and $u$ in $C_{\gamma}$. By the uniqueness of joins and the 
	definition of $u_1$ it follows that $u_1=u'$. Thus $c$ is the lexicographically 
	smallest maximal chain in $[u,v]_{\gamma}$. Finally, Lemma~\ref{lem:i_0} implies that $c$ is the 
	unique maximal rising chain.
\end{proof}

\begin{remark}
	In the case where $W$ is finite and crystallographic, Colin Ingalls and Hugh Thomas have shown that
	$C_{\gamma}$ is trim. Trimness is a lattice property that generalizes distributivity to ungraded
	lattices. Then, by definition of trimness, it follows that $C_{\gamma}$ is left-modular, meaning 
	that there exists a maximal chain 
	$c:x_{1}\lessdot_{\gamma} x_{2}\lessdot_{\gamma}\cdots\lessdot_{\gamma} x_{n}$ satisfying
	$(y\vee_{\gamma} x_{i})\wedge_{\gamma} z = y\vee_{\gamma}(x_{i}\wedge_{\gamma}z)$,
	for all $y<_{\gamma}z$ and $i\in\{1,2,\ldots,n\}$. According to \cite{liu99left}, this property
	yields another EL-labeling of $C_{\gamma}$, defined by
	\begin{displaymath}
		\xi(y,z)=\min\{i\mid y\vee_{\gamma} x_{i}\wedge_{\gamma} z=z\},
	\end{displaymath}
	for all $y,z\in L$ with $y\lessdot_{\gamma}z$. It is not hard to show that this labeling is 
	structurally different from our labeling.
\end{remark}

\begin{proof}[Proof of Theorem~\ref{thm:infinite_shellability}]
	This follows by definition from Theorem~\ref{thm:el_labeling}. 
\end{proof}

\begin{remark}
  \label{rem:tree}
	In the case where $W$ is finite, \cite{reading07clusters}*{Remark~2.1}, states that the 
	$\gamma$-sortable elements constitute a spanning tree of the Hasse diagram of $C_{\gamma}$, which is 
	rooted at the identity. The edges of this spanning tree correspond to covering relations 
	$u\lessdot_{\gamma}v$ in $C_{\gamma}$ such that $u$ is a prefix of $v$. This spanning tree is 
	related to the labeling $\lambda_{\gamma}$ in the following way: let $w\in W$, with $\ell_{S}(w)=k$, 
	and let $(i_{0},i_{1},\ldots,i_{k-1})$ be the sequence of edge-labels of the unique rising chain in 
	$[\varepsilon,w]_{\gamma}$. In view of Theorem~\ref{thm:el_labeling}, and 
	\cite{reading07clusters}*{Remark~2.1}, we notice that the unique path from $\varepsilon$ to $w$ in 
	the spanning tree of $C_{\gamma}$ corresponds to the unique rising chain in 
	$[\varepsilon,w]_{\gamma}$. Hence, the $\gamma$-sorting word of $w$ is 
	$s_{i_{0}}s_{i_{1}}\cdots s_{i_{k-1}}$, where $s_{i_{j}}$ is the $i_{j}$-th letter of 
	$\gamma^{\infty}$, and the length of the rising chain in $[\varepsilon,w]_{\gamma}$ is
	precisely $\ell_{S}(w)$. Moreover, it follows from the proof of Theorem~\ref{thm:el_labeling} 
	that the length of the unique rising chain in an interval $[u,v]_{\gamma}$ equals 
	$\ell_S(v)-\ell_S(u)$. 
	
	In view of Theorem~\ref{thm:el_labeling}, we can carry out the same construction even in the case of 
	infinite Coxeter groups. 
\end{remark}

\section{Applications}
  \label{sec:applications}
In \cite{reading05lattice}, Nathan Reading investigated, among others, the topological properties of open
intervals in so-called \emph{fan posets}. A fan poset is a certain partial order defined on the maximal 
cones of a complete fan of regions of a real hyperplane arrangement. For a finite Coxeter group $W$ and a
Cambrian congruence $\theta$, the \emph{Cambrian fan $\mathcal{F}_{\theta}$} is the complete fan induced by 
certain cones in the Coxeter arrangement $\mathcal{A}_{W}$ of $W$. More precisely, each such cone is a 
union of regions of $\mathcal{A}_{W}$ which correspond to elements of $W$ lying in the same congruence 
class of $\theta$. It is the assertion of \cite{reading05lattice}*{Theorem~1.1}, that a Cambrian lattice of 
$W$ is the fan poset associated to the corresponding Cambrian fan. The following theorem is a concatenation
of \cite{reading05lattice}*{Theorem~1.1} and \cite{reading05lattice}*{Propositions~5.6~and~5.7}. In fact, 
Propositions~5.6 and 5.7 in \cite{reading05lattice} imply this result for a much larger class of fan posets. 

\begin{theorem}
  \label{thm:finite_topology}
	Let $W$ be a finite Coxeter group and let $\gamma\in W$ be a Coxeter element. Every open interval in 
	the Cambrian lattice $C_{\gamma}$ is either contractible or spherical.
\end{theorem}

It is well-known that the reduced Euler characteristic of the order complex
of an open interval $(x,y)$ in a poset determines $\mu(x,y)$, see for instance
\cite{stanley01enumerative}*{Proposition~3.8.6}. Hence, it follows immediately from
Theorem~\ref{thm:finite_topology} that for $\gamma$-sortable elements $x$ and $y$ in a finite Coxeter 
group $W$ satisfying $x\leq_{\gamma} y$, we have $\lvert\mu(x,y)\rvert\leq 1$, as was already remarked in 
\cite{reading06cambrian}*{pp.\;4-5}. In light of Proposition~\ref{prop:mobius} and 
Theorem~\ref{thm:el_labeling}, we can extend this statement to compute the M\"obius function of closed 
intervals in the Cambrian semilattice $C_{\gamma}$, by counting the falling maximal chains with respect to 
the labeling defined in \eqref{eq:labeling}, as our next theorem shows.

\begin{theorem}
  \label{thm:cambrian_mobius}
	Let $W$ be a (possibly infinite) Coxeter group and $\gamma\in W$ a Coxeter element. For 
	$u,v\in C_{\gamma}$ with $u\leq_{\gamma} v$, we have $|\mu(u,v)|\leq 1$. 
\end{theorem}
\begin{proof}
	In view of Proposition~\ref{prop:mobius} it is enough to show that the interval $[u,v]_{\gamma}$ 
	has at most one maximal falling chain. We use similar arguments as in 
	the proof of Theorem~\ref{thm:el_labeling} and proceed by induction on length and rank. Again, 
	we may assume that $\ell_{S}(v)=k\geq 3$ and that $W$ is a Coxeter group of rank $n\geq 2$, since 
	the result is trivial otherwise. Suppose that the induction hypothesis is true for all parabolic 
	subgroups of $W$ with rank $<n$ and suppose that for every closed interval $[u',v']_{\gamma}$ of 
	$C_{\gamma}$ with $\ell_S(v')<k$, there exists at most one falling maximal chain. We will show that
	there is at most one maximal falling chain in the interval $[u,v]_{\gamma}$ as well. For $s$ initial 
	in $\gamma$, we distinguish two cases: (1) $s\not\leq_{\gamma} v$ and (2) $s\leq_{\gamma} v$. 

	(1) The result follows directly by induction on the rank of $W$ by following the steps 
	of case (1) in the proof of Theorem~\ref{thm:el_labeling}.

	(2a) Suppose in addition that $s\leq_{\gamma} u$. The result follows directly by induction on the 
	length of $v$ by following the steps of case (2a) in the proof of Theorem~\ref{thm:el_labeling}.
	
	(2b) Suppose now that $s\not\leq_{\gamma} u$.  It follows from Lemma~\ref{lem:i_0} that a maximal 
	chain $u=x_0\lessdot_{\gamma}x_1\lessdot_{\gamma}\cdots\lessdot_{\gamma}x_{t-1}\lessdot x_t=v$ of 
	$[u,v]_{\gamma}$ can be falling only if $\lambda_{\gamma}(x_{t-1},v)=1$. Hence, if there is no 
	element $v_1\in (u,v)_{\gamma}$, with $v_1\lessdot v$ satisfying $\lambda_{\gamma}(v_1,v)=1$, then 
	the interval $[u,v]_{\gamma}$ has no maximal falling chain, which means that $\mu(u,v)=0$. Otherwise, 
	consider the interval $[u,v_1]_{\gamma}$. By the choice of $v_1$, it follows that every maximal 
	falling chain in $[u,v_{1}]_{\gamma}$ can be extended to a maximal falling chain in the interval 
	$[u,v]_{\gamma}$. Conversely, every maximal falling chain in $[u,v]_{\gamma}$ can be restricted to a 
	maximal falling chain in $[u,v_{1}]_{\gamma}$. Therefore, since $\ell_S(v_1)<\ell_S(v)$, we deduce
	from the induction hypothesis that the interval $[u,v_1]_{\gamma}$ has at most one maximal falling 
	chain. Thus $|\mu(u,v)|\leq 1$. 
\end{proof}

Again in view of Lemma~\ref{lem:well_defined} the statement of Theorem~\ref{thm:cambrian_mobius} does
not depend on a reduced word for $\gamma$, even though our labeling does.

In addition Propositions~5.6 and 5.7 in \cite{reading05lattice} characterize the open 
intervals in a (finite) Cambrian lattice which are contractible, and those which are spherical in the 
following way: an interval $[u,v]_{\gamma}$ in $C_{\gamma}$ is called \emph{nuclear} if the 
join of the upper covers of $u$ is precisely $v$. Nathan Reading showed that the nuclear intervals are 
precisely the spherical intervals. With the help of our labeling, we can generalize this characterization
to infinite Coxeter groups. 

\begin{theorem}
  \label{thm:mobius_nuclear}
	Let $u,v\in C_{\gamma}$ with $u\leq_{\gamma}v$ and let $k$ denote the number of atoms of the interval 
	$[u,v]_{\gamma}$. Then, $\mu(u,v)=(-1)^{k}$ if and only if $[u,v]_{\gamma}$ is nuclear.
\end{theorem}

For the proof of Theorem~\ref{thm:mobius_nuclear}, we need the following lemma.

\begin{lemma}
  \label{lem:nuclear_sub}
	Let $u,v\in C_{\gamma}$ with $u\leq_{\gamma}v$, and let $s$ be initial in $\gamma$. 
	Suppose further that $s\not\leq_{\gamma}u$, while $s\leq_{\gamma}v$. 
	Then the following are equivalent:
	\begin{enumerate}
	 \item The interval $[u,v]_{\gamma}$ is nuclear.
	 \item There exists an element $v'\in [u,v]_{\gamma}$ satisfying 
	 $s\not\leq_{\gamma}v'\lessdot_{\gamma}v$, and the interval $[u,v']_{\gamma}$ is nuclear.
	\end{enumerate}
\end{lemma}
\begin{proof}
	Let $A=\{w\in C_{\gamma}\mid u\lessdot_{\gamma}w\leq_{\gamma}v\}$ be the set of
	atoms of the interval $[u,v]_{\gamma}$. 
	Since $s\leq_{\gamma}v$ and $u\leq_{\gamma}v$, we conclude from Theorem~\ref{thm:joins_meets} that 
	the join $s\vee_{\gamma}u$ exists, and we set $z=s\vee_{\gamma}u$. It follows from 
	Lemma~\ref{lem:joincov} that $u\lessdot_{\gamma}z$, and hence $z\in A$.
	We set $A_z=A\setminus\{z\}$ and remark that if $w\in A_z$, then $s\not\leq_{\gamma} w$. 
	Indeed, suppose that there exists some $z'\in A_z$ with $s\leq_{\gamma}z'$. Since $u\lessdot_{\gamma}z'$, 
	this implies $s\vee_{\gamma}u\leq_{\gamma}z'$, and hence $z\leq_{\gamma} z'$. Since $z$ and $z'$ 
	both cover $u$, this implies $z=z'$, which contradicts $z\notin A_z$. Thus, $s\not\leq_{\gamma} w$ 
	for all $w\in A_z$. In particular we have $A_z\subseteq W_{\langle s\rangle}$.
	
	(1)$\Rightarrow$(2) Suppose that $[u,v]_{\gamma}$ is nuclear and let $v'=\bigvee A_z$. 
	Again, Theorem~\ref{thm:joins_meets} ensures that $v'$ exists 
	and that it satisfies $u\leq_{\gamma}v'\leq_{\gamma}v$. 
	Since $A_z\subseteq W_{\langle s\rangle}$, it follows from 
	\cite{reading11sortable}*{Proposition~2.20} that $v'=\bigvee A_{z}\in W_{\langle s\rangle}$
	which means that $s\not\leq_{\gamma} v'$, and $A_z$ is thus the set of atoms 
	of the interval $[u,v']_{\gamma}$. Hence, $[u,v']_{\gamma}$ is nuclear.
	It remains to show that $v'\lessdot_{\gamma}v$. It follows from $u\leq_{\gamma}v'$ and the 
	associativity of $\vee_{\gamma}$ that
	\begin{displaymath}
		v=\bigvee A=z\vee_{\gamma}\left(\bigvee A_{z}\right)=z\vee_{\gamma}v'
		  =(s\vee_{\gamma}u)\vee_{\gamma}v'=s\vee_{\gamma}(u\vee_{\gamma}v')=s\vee_{\gamma}v'.
	\end{displaymath}
	From above, we know that $s¸\not\leq_{\gamma}v'$ and we can apply Lemma~\ref{lem:joincov} which 
	implies immediately that $v'\lessdot_{\gamma}s\vee_{\gamma}v'=v$. 

	(2)$\Rightarrow$(1) Suppose now that there exists an element $v'\in [u,v]_{\gamma}$ satisfying 
	$s\not\leq_{\gamma}v'\lessdot_{\gamma}v$, and suppose that the interval $[u,v']_{\gamma}$ is nuclear. 
	Let $A'$ denote the set of atoms of $[u,v']_{\gamma}$. Since $s\not\leq_{\gamma}v'$ and $s\leq_{\gamma} z$, 
	it follows that $z\notin A'$, thus 
	$A'\subseteq A_z$. Furthermore, from $s\leq_{\gamma}v$,  
	$v'\lessdot_{\gamma}v$ and Lemma \ref{lem:joincov} we deduce that $s\vee_{\gamma}v'=v$. Now we have 
	\begin{displaymath}
		z\vee_{\gamma}v'=(s\vee_{\gamma}u)\vee_{\gamma}v'
		  =s\vee_{\gamma}(u\vee_{\gamma}v')=s\vee_{\gamma}v'=v,
	\end{displaymath}
	since $u\leq_{\gamma}v'$. Thus, we can write $v=\bigvee\bigl(A'\cup\{z\}\bigr)$. 
	Finally, we will show that $v=\bigvee A$.  
	Let $z'\in A\setminus A'$. Since $z'\leq_{\gamma}v$, it follows that 
	\begin{displaymath}
		\bigvee\bigl(A'\cup\{z,z'\}\bigr)=\bigvee\bigl(A'\cup\{z\}\bigr)\vee_{\gamma}z'
		  =(v'\vee_{\gamma}z)\vee_{\gamma}z'=v\vee_{\gamma}z'=v,
	\end{displaymath}
	and hence $v=\bigvee A$. This implies that $[u,v]_{\gamma}$ is nuclear. 
\end{proof}

We remark that under the hypothesis of Lemma~\ref{lem:nuclear_sub}, the element $v'=\bigvee A_z$ constructed in 
the part (1)$\Rightarrow$(2) of the proof is the unique element in $[u,v]_{\gamma}$ satisfying condition (2). 
The uniqueness of $v'$ is a consequence of the uniqueness of the join $\bigvee A_z$.

\begin{proof}[Proof of Theorem~\ref{thm:mobius_nuclear}]
	In view of Proposition~\ref{prop:mobius}, we need to show that $[u,v]_{\gamma}$ has a falling chain 
	if and only if $[u,v]_{\gamma}$ is nuclear. We use similar arguments as in the proof of 
	Theorem~\ref{thm:el_labeling} and proceed by induction on length and rank. Again we may assume that
	$\ell_{S}(v)=k\geq 3$ and that $W$ is a Coxeter group of rank $n\geq 2$, since the result is trivial
	otherwise. Suppose that the induction hypothesis is true for all parabolic subgroups of $W$ with 
	rank $<n$ and suppose that for every closed interval $[u',v']_{\gamma}$ of $C_{\gamma}$ with 
	$\ell_{S}(v')<k$ there exists a falling maximal chain if and only if $[u',v']_{\gamma}$ is nuclear.
	For $s$ initial in $\gamma$, we distinguish two cases: (1) $s\not\leq_{\gamma}v$ and
	(2) $s\leq_{\gamma}v$.
	
	(1) The result follows directly by induction on the rank of $W$ by following the steps of case 
	(1) in the proof of Theorem~\ref{thm:cambrian_mobius}. 

	(2a) Suppose in addition that $s\leq_{\gamma}u$. The result follows directly by induction on the 
	length of $v$ by following the steps of case (2a) in the proof of Theorem~\ref{thm:cambrian_mobius}.
	
	(2b) Suppose now that $s\not\leq_{\gamma}u$. If $[u,v]_{\gamma}$ is nuclear, then
	Lemma~\ref{lem:nuclear_sub} implies that there exists a unique element $v'\in C_{\gamma}$ with
	$u\leq_{\gamma}v'\lessdot_{\gamma}v$ such that $[u,v']_{\gamma}$ is nuclear, and 
	$s\not\leq_{\gamma}v'$. Thus, we can apply induction on the rank of $W$ and obtain a maximal falling 
	chain $c':u=x_{0}\lessdot_{\gamma}x_{1}\lessdot_{\gamma}\cdots\lessdot_{\gamma}x_{t-1}=v'$. 
	Lemma~\ref{lem:i_0} implies that $1\notin\lambda_{\gamma}(c')$, and 
	Lemma~\ref{lem:labeling_recursion} implies that 
	$\lambda_{\gamma}(v',v)=1$. Thus, the chain $c:u=x_{0}\lessdot_{\gamma}x_{1}\lessdot_{\gamma}\cdots
	\lessdot_{\gamma}x_{t-1}\lessdot_{\gamma}x_{t}=v$ is a falling maximal chain in $[u,v]_{\gamma}$, 
	and Theorem~\ref{thm:cambrian_mobius} implies its uniqueness.

	Conversely, suppose that there exists a maximal falling chain 
	$c:u=x_{0}\lessdot_{\gamma}x_{1}\lessdot_{\gamma}\cdots\lessdot_{\gamma}x_{t}=v$ in $[u,v]_{\gamma}$,
	and let $A=\{w\in C_{\gamma}\mid u\lessdot_{\gamma}w\;\mbox{and}\;w\leq_{\gamma}v\}$ denote the set
	of atoms of $[u,v]_{\gamma}$. In view of Lemma~\ref{lem:i_0}, we notice that 
	$\lambda_{\gamma}(x_{t-1},v)=1$, which implies $s\not\leq_{\gamma}x_{t-1}$. Clearly  
	$\ell_{S}(x_{t-1})<k$ and the chain 
	$c':u=x_{0}\lessdot_{\gamma}x_{1}\lessdot_{\gamma}\cdots\lessdot_{\gamma}x_{t-1}$ is falling, 
	thus by induction we can conclude that the interval $[u,x_{t-1}]_{\gamma}$ is nuclear. 
	Since $s\not\leq_{\gamma}x_{t-1}\lessdot_{\gamma} v$, it follows from Lemma \ref{lem:nuclear_sub} 
	that $[u,v]_{\gamma}$ is nuclear. This completes the proof of the theorem. 
\end{proof}

\begin{proof}[Proof of Theorem~\ref{thm:cambrian_topology}]
	Theorem~\ref{thm:infinite_shellability} implies that every closed interval $[u,v]_{\gamma}$ of 
	$C_{\gamma}$ is EL-shellable. Theorem~5.9 in \cite{bjorner96shellable} states that the dimension 
	of the $i$-th homology group of the order complex of $(u,v)_{\gamma}$ corresponds to the number of 
	falling chains in $[u,v]_{\gamma}$ having length $i+2$. Theorem~\ref{thm:cambrian_mobius} implies 
	that there is at most one falling chain in $[u,v]_{\gamma}$. Hence, either all homology groups of 
	the order complex of $(u,v)_{\gamma}$ have dimension $0$ (then, $(u,v)_{\gamma}$ is contractible) or 
	there exists exactly one homology group of dimension $1$ (then, $(u,v)_{\gamma}$ is spherical). 
	Finally, the characterization of the spherical intervals  	
	is an immediate consequence of Theorem \ref{thm:cambrian_mobius}. 
\end{proof}

\begin{remark}
	Christian Stump (private conversation) pointed out that, in the case of finite Coxeter groups, the 
	statements of Theorems~\ref{thm:infinite_shellability} and \ref{thm:cambrian_topology} can be 
	generalized straightforward to the increasing flip order of subword complexes for so-called 
	\emph{realizing words}. In \cite{pilaud12brick}*{Section~5.3}, Pilaud and Stump defined an acyclic, 
	directed, edge-labeled graph on the facets of the subword complex, the so-called 
	\emph{increasing flip graph}. The transitive closure of this graph is then a partial order, the 
	\emph{increasing flip order}. In the case of realizing words, the Hasse diagram of the increasing 
	flip order coincides with the increasing flip graph which then yields an edge-labeling of this poset. 
	One can show that this labeling is indeed an EL-labeling and that every interval has at most one 
	falling chain. This has recently been done in \cite{pilaud12el}.  
	
	It is the statement of \cite{pilaud12brick}*{Corollary~6.31} that the Cambrian lattices of finite
	Coxeter groups correspond to the increasing flip order of special subword complexes. In addition, 
	the construction of \cite{pilaud12brick} as briefly described in the previous paragraph provides a 
	nice geometric description of the statements of Theorems~\ref{thm:infinite_shellability} and
	\ref{thm:cambrian_topology}.
\end{remark}

We conclude this section with a short example of an infinite Coxeter group.

\begin{example}
	Consider the affine Coxeter group $\tilde{A}_{2}$, which is generated by the set
	$\{s_{1},s_{2},s_{3}\}$ satisfying $(s_{1}s_{2})^{3}=(s_{1}s_{3})^{3}=(s_{2}s_{3})^{3}=\varepsilon$,
	as well as $s_{1}^{2}=s_{2}^{2}=s_{3}^{2}=\varepsilon$. Consider the Coxeter element 
	$\gamma=s_{1}s_{2}s_{3}$. Figure~\ref{fig:camb_affine_a2} shows the sub-semilattice of the 
	Cambrian semilattice $C_{\gamma}$ consisting of all $\gamma$-sortable elements of $\tilde{A}_{2}$ of 
	length $\leq 7$. We encourage the reader to verify Theorem~\ref{thm:el_labeling} and 
	Theorem~\ref{thm:cambrian_mobius}. 
	
	\begin{figure}
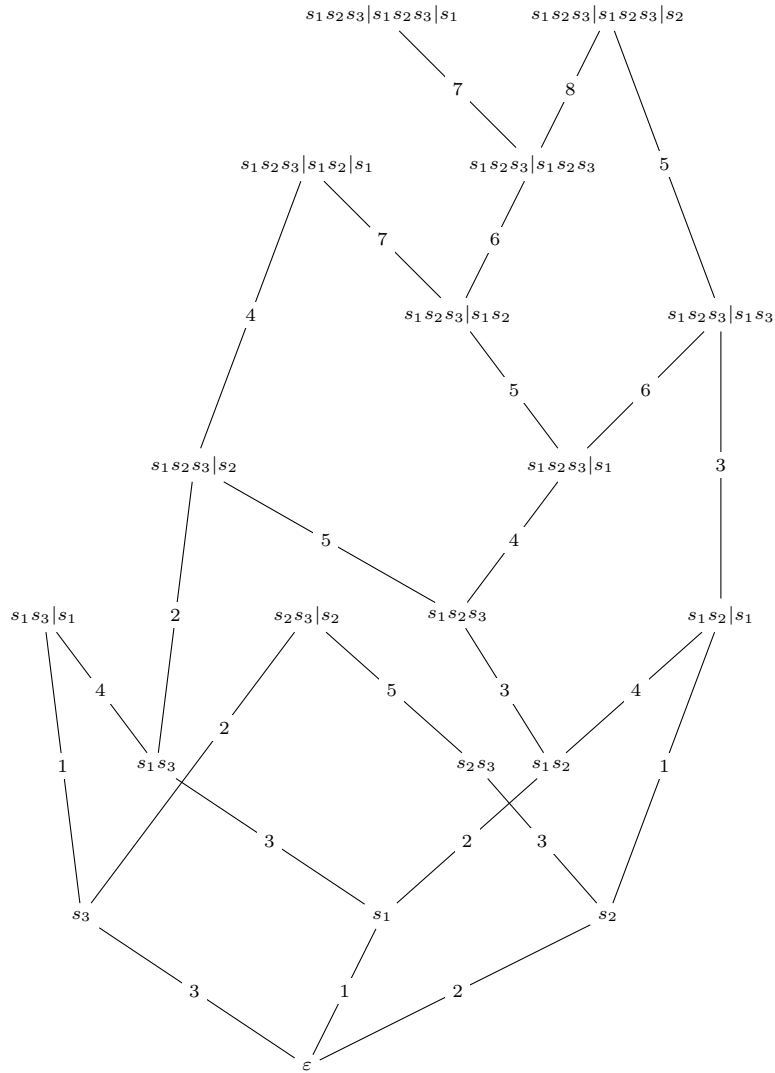

		\centering
		\cambAffineATwo{1}{2}
		\caption{The first seven ranks of an $\tilde{A}_{2}$-Cambrian semilattice, with the labeling
		  as defined in \eqref{eq:labeling}.}
		\label{fig:camb_affine_a2}
	\end{figure}
\end{example}

\section*{Acknowledgements}
The authors would like to thank Nathan Reading for pointing out a weakness in a previous version of the proof
of Theorem~\ref{thm:el_labeling} and for many helpful discussions on the topic.

\bibliography{literature}

\end{document}